\documentclass[10pt]{article}

\usepackage[english]{babel}
\usepackage[utf8]{inputenc}

\usepackage{amsmath}
\usepackage{amssymb}
\usepackage{amsthm}

\usepackage[centertableaux]{ytableau}
\ytableausetup{boxsize=0.5em}

\usepackage{array}
\usepackage{caption}
\newcolumntype{C}[1]{>{\centering\arraybackslash}m{#1}}   
\newcolumntype{M}{>{$}c<{$}}   

\usepackage{mathtools}

\usepackage[numbers]{natbib}
\usepackage[pdfpagelabels=true]{hyperref}
\hypersetup{%
pdfauthor = {Johannes Hahn},
pdftitle = {W-graphs and Gyoja's W-graph algebra}
}

\usepackage{tikz}
\tikzset
{
	desc/.style=
	{
		fill=white,inner sep=2pt,font=\scriptsize
	},
	Vertex/.style =
	{
		inner sep = 1pt,
		outer sep = 2pt,
		minimum size=15pt,
		circle,
		draw=black!70,
		thick,
		font=\scriptsize
	},
	EdgeT/.style =
	{
		ultra thick,
		draw=black,
		font=\scriptsize
	},
	EdgeI/.style =
	{
		->,
		draw=black!70,
		font=\scriptsize
	}
}

\newtheorem{theorem}{Theorem}
\newtheorem{lemma}[theorem]{Lemma}
\newtheorem{corollary}[theorem]{Corollary}

\newtheorem{definition}[theorem]{Definition}
\newtheorem{remark}[theorem]{Remark}
\newtheorem{example}[theorem]{Example}
\newtheorem{conjecture}[theorem]{Conjecture}

\newcommand{\IN}{\mathbb{N}}
\newcommand{\IZ}{\mathbb{Z}}
\newcommand{\IC}{\mathbb{C}}

\newcommand{\abs}[1]{{\left\lvert#1\right\rvert}}
\newcommand{\ceil}[1]{\left\lceil #1 \right\rceil}
\newcommand{\floor}[1]{\left\lfloor #1 \right\rfloor}
\newcommand{\Set}[1]{{\left\lbrace#1\right\rbrace}}

\newcommand{\isomorphic}{\cong}

\DeclareMathOperator{\sgn}{sgn}
\DeclareMathOperator{\Hom}{Hom}
\DeclareMathOperator{\im}{im}
\DeclareMathOperator{\ord}{ord}

\DeclareMathOperator{\Irr}{Irr}
\DeclareMathOperator{\rad}{rad}

\title{$W$-graphs and Gyoja's $W$-graph algebra \footnotemark[1] }
\author{Johannes Hahn \footnotemark[2] }

\begin{document}
\maketitle

\footnotetext[1]{MSC classes: 20C08, 20F55, 16G30}
\footnotetext[2]{Formerly Friedrich-Schiller-University Jena; Contact information -- eMail: \href{mailto:johannes@hahn-rostock.de}{johannes@hahn-rostock.de}, Telephone: (+49)1781472539 (Germany)}

\begin{abstract}
Let $(W,S)$ be a finite Coxeter group. Kazhdan and Lusztig introduced the concept of $W$-graphs and Gyoja proved that every irreducible representation of the Iwahori-Hecke algebra $H(W,S)$ can be realized as a $W$-graph. Gyoja defined an auxiliary algebra for this purpose which  -- to the author's best knowledge -- was never explicitly mentioned again in the literature after Gyoja's proof (although the underlying ideas were reused). The purpose of this paper is to resurrect this $W$-graph algebra and study its structure and its modules. A new explicit description of it as a quotient of a certain path algebra is given. A general conjecture is proposed that -- if it turns out to be true -- would imply strong restrictions on the structure of $W$-graphs. This conjecture is then proven for Coxeter groups of type $I_2(m)$, $B_3$ and $A_1$ -- $A_4$.
\end{abstract}

\section{Introduction}

Let $(W,S)$ be a finite Coxeter group. Kazhdan and Lusztig introduced $W$-graphs in \cite{KL} in an attempt to capture certain combinatorial features of Kazhdan-Lusztig-cells and of the cell representations associated to them. By definition every cell representation is a $W$-graph representation. The converse is not true.

Perhaps the most general result in that regard Gyoja proved that every irreducible  representation (and hence every reducible representation as well) of the Hecke algebra $H(W,S)$ can be realized as a $W$-graph representation if $W$ is finite (see \citep[2.3.(1)]{Gyoja}). In that proof the Iwahori-Hecke algebra is embedded into a larger algebra, which I will denote $\Omega$ in this paper, and it is proven that there exists a left inverse of this embedding. The $W$-graph algebra $\Omega$ is constructed in such a way that its modules correspond to $W$-graphs (up to choice of an appropriate basis). Using any one of these left inverses, every $H$-module can be considered as an $\Omega$-module and the result follows.

Gyoja's proof is non-constructive as it does not provide a concrete left inverse of the embedding $H\hookrightarrow\Omega$ and does not offer additional information about the $W$-graphs that were constructed in this fashion or any information about general $W$-graphs. In my thesis \cite{hahn2013diss} I discovered that a careful analysis of $\Omega$ reveals a fine structure that gives much more detailed information about $W$-graphs. An explicit left inverse utilising Lusztig's asymptotic algebra is also provided in \citep[Satz 4.3.2]{hahn2013diss}.

The starting point for this analysis is the observation that $\Omega$ is a quotient of a path algebra over a quiver which is describable entirely in terms of the Dynkin diagram, a fact that is implicitly contained in Gyoja's paper but was not interpreted in that way. Gyoja's definition \citep[2.5]{Gyoja} gives elements of $\Omega$ which basically realize the vertex idempotents and the edge elements of a path algebra (this will be made precise in Lemma \ref{lemma:Xi_path_alg}). The first main result of my paper is to give an explicit set of relations for this quotient (Theorem \ref{theorem:Omega_relations}). The relations are inspired by the work of Stembridge \cite{Stembridge2008admissble} where similar equations appear for the edge weights of so called ``admissible'' $W$-graphs although they were neither formulated for general $W$-graphs nor interpreted as relations for an underlying algebra. This set of relations seems to be different from the presentation Gyoja gives in the appendix of his paper.

\medbreak
Once this new presentation of $\Omega$ is established it will be applied to breaking down the structure of $\Omega$ further. At the moment this is only done for some small Coxeter groups by a case-by-case analysis but the proofs are so similar in spirit that I proposed a general conjecture in my thesis whose essence is that $\Omega$ should also be a quotient of a generalized path algebra over a different quiver which should have $\Irr(W)$ as its vertex set and should be acyclic. The algebras associated to the vertices should be matrix algebras.

In the cases for which the conjecture is true, it has several important consequences like the following:
\begin{itemize}
	\item $k\Omega$ is finitely generated as a $k$-module where $k$ is a so called good ring for $(W,S)$, i.e. a ring $k\subseteq\mathbb{C}$ with $2\cos(\tfrac{2\pi}{m_{st}})\in k$ for all $s,t\in S$ and $p\in k^\times$ for all bad primes $p$. (See \citep[table 1.4]{geckjacon} for a detailed description of what that means for each type of finite Coxeter group.)
	\item The Jacobson radical $\rad(k\Omega)$ is finitely generated by an explicitly describable finite list of elements and $k\Omega/\rad(k\Omega) \isomorphic \prod_{\lambda\in\Irr(W)} k^{d_\lambda\times d_\lambda}$ where $d_\lambda$ denotes the degree of the irreducible character $\lambda$. This implies that Gyoja's conjecture (c.f. \citep[2.18]{Gyoja}) holds.
	\item There is an enumeration $\lambda_1,\ldots,\lambda_n$ of $\Irr(W)$ such that every $k\Omega$-module $V$ has a natural filtration
	\[0 = V^0 \subseteq V^1 \subseteq \ldots \subseteq V^n = V\]
	which realizes the decomposition of $V$ into irreducibles in the sense that $V^i / V^{i-1}$ is isomorphic to a direct sum of irreducibles of isomorphism class $\lambda_i$.
\end{itemize}

Because of the last consequence I named the conjecture ``$W$-graph decomposition conjecture''. The second consequence and in particular the connection to Gyoja's conjecture was my original motivation for investigating the $W$-graph algebra and its fine structure. By the time of writing, the decomposition conjecture has been proven for Coxeter groups of types $A_1$--$A_4$, $I_2(m)$ and $B_3$.

\medbreak
The paper is organized as follows: The first section introduces some notation, recalls the definition of $W$-graphs (following \cite{geckjacon} which is slightly more general than Kazhdan's and Lusztig's), the definition of the $W$-graph algebra (following \cite{Gyoja} though with a different notation) and proves some basic lemmas establishing the connection between $W$-graphs and $\Omega$-modules. Section \ref{section:relations} is devoted to stating and proving an explicit description of $\Omega$ in terms of generators and relations which will be the basis for all subsequent proofs. Section \ref{section:decomposition_conjecture} contains the statement of the decomposition conjecture and a short discussion of its consequences, while Section \ref{section:proof_of_conjecture} is devoted to the proofs of the conjecture for small Coxeter groups.

\section{Preliminaries}

\subsection{Notation}

Throughout the paper fix a finite Coxeter system $(W,S)$. The Iwahori-Hecke algebra $H=H(W,S)$ of $(W,S)$ is the $\mathbb{Z}[v^{\pm 1}]$-algebra (where $v$ is an indeterminate) which is freely generated by $(T_s)_{s\in S}$ subject only to the relations
\[\forall s\in S: T_s^2 = 1 + (v-v^{-1}) T_s\quad\textrm{and}\]
\[\forall s,t\in S: \Delta_{m_{st}}(T_s,T_t)=0\]
where $m_{st}$ denotes the order of $st\in W$ and $\Delta_m(x,y)$ is the \emph{$m$-th braid commutator} of ring elements $x$ and $y$ which is defined as follows
\[\Delta_m(x,y) := \underbrace{xyx\ldots}_{m\,\textrm{factors}} - \underbrace{yxy\ldots}_{m\,\textrm{factors}}\]
In particular $\Delta_0(x,y)=0$, $\Delta_1(x,y)=x-y$, $\Delta_2(x,y)=xy-yx$, $\Delta_3(x,y)=xyx-yxy$ and so on.

\medbreak
Also fix a good ring for $(W,S)$, that is, a ring $k\subseteq\mathbb{C}$ with $2\cos(\tfrac{2\pi}{m_{st}})\in k$ for all $s,t\in S$ and $p\in k^\times$ for all so called bad primes $p$. (See \citep[table 1.4]{geckjacon} for a detailed description of what that means for each type of finite Coxeter groups.)

A ring is good if it is ``big enough'' for the purposes of representation theory of Coxeter groups. For example every good field is a splitting field for $W$.

\medbreak
%
If $A$ is a $k$-algebra and $k'$ is a commutative $k$-algebra, then $k'A$ will be used as shorthand for the $k'$-algebra $k'\otimes_k A$. Similarly the abbreviation $k'V$ will be used for the $k'A$-module $k'\otimes_k V$ if $V$ is an $A$-module.

\subsection{$W$-graphs}

\begin{definition}[c.f. \cite{KL} and \cite{geckjacon}]
A \emph{$W$-graph with edge weights in $k$} is a triple $(\mathfrak{C},\mathcal{I},m)$ consisting of a finite set $\mathfrak{C}$ of \emph{vertices}, a \emph{vertex labelling} map $\mathcal{I}: \mathfrak{C}\to \{I \mid I\subseteq S\} $ and a family of \emph{edge weight} matrices $m^s\in k^{\mathfrak{C}\times\mathfrak{C}}$ for $s\in S$ (here $k^{\mathfrak{C}\times\mathfrak{C}}$ denotes the ring of matrices whose rows and columns are indexed with $\mathfrak{C}$ and whose entries are elements of $k$) such that the following conditions hold:
\begin{enumerate}
	\item $\forall x,y\in\mathfrak{C}: m_{xy}^s \neq 0 \implies s\in \mathcal{I}(x)\setminus \mathcal{I}(y)$.
	\item The matrices
	\[\omega(T_s)_{xy} := \begin{cases} -v^{-1}\cdot 1_k & \textrm{if}\;\;x=y, s\in \mathcal{I}(x) \\ v\cdot 1_k & \textrm{if}\;\;x=y, s\notin \mathcal{I}(x) \\ m_{xy}^s & \textrm{otherwise}\end{cases}\]
	induce a matrix representation $\omega: k[v^{\pm1}]H\to k[v^{\pm1}]^{\mathfrak{C}\times\mathfrak{C}}$.
\end{enumerate}
The associated directed graph is defined as follows: The vertex set is $\mathfrak{C}$ and there is a directed edge $x\leftarrow y$ if and only if $m_{xy}^s\neq 0$ for some $s\in S$. If this is the case, then the value $m_{xy}^s$ is called the \emph{weight} of the edge. The set $I(x)$ is called the \emph{vertex label} of $x$.
\end{definition}

Note that condition 1 and the definition of $\omega(T_s)$ already guarantees $\omega(T_s)^2=1+(v-v^{-1})\omega(T_s)$ so that the only non-trivial requirement in condition 2 is the braid relation $0=\Delta_{m_{st}}(\omega(T_s),\omega(T_t))$.

The definition seems to allow up to $|I(x)\setminus I(y)|$ different edge weights for a single edge $x\leftarrow y$. We will prove later that all values $m_{xy}^s$ with $s\in I(x)\setminus I(y)$ are in fact equal.

\bigbreak
Given a $W$-graph as above the matrix representation $\omega$ turns the space $k[v^{\pm 1}]^{\mathfrak{C}}$ of columns vectors indexed with $\mathfrak{C}$ with entries in $k[v^{\pm 1}]$ into a left module for the Hecke algebra $k[v^{\pm 1}]H$. It is natural to ask whether a converse is true. In situations where the Hecke algebra is split semisimple the answer is yes as shown by Gyoja:
\begin{theorem}[c.f. \cite{Gyoja}]
Let $K\subseteq\mathbb{C}$ be a splitting field for $W$. Every irreducible representation of $K(v)H$ can be realized as a $W$-graph module for some $W$-graph with edge weights in $K$.
\end{theorem}

\subsection{Gyoja's $W$-graph algebra}

\begin{definition}
Define $\Xi$ as the $\mathbb{Z}$-algebra that is freely generated by $e_s, x_s$ for $s\in S$ with respect to the following relations:
\begin{enumerate}
	\item $\forall s\in S: e_s^2=e_s$
	\item $\forall s,t\in S: e_s e_t = e_t e_s$
	\item $\forall s\in S: e_s x_s=x_s,\; x_s e_s = 0$.
\end{enumerate}

Furthermore define
\[\iota(T_s) := -v^{-1} e_s + v(1-e_s) + x_s \in \mathbb{Z}[v^{\pm1}]\Xi\]
for all $s\in S$. The braid commutator $\Delta_{m_{st}}(\iota(T_s),\iota(T_t))$ can be written as $\sum_{\gamma\in\mathbb{Z}} y^\gamma(s,t) v^\gamma$ with uniquely determined elements $y^\gamma(s,t)\in\Xi$.

The \emph{$W$-graph algebra $\Omega$} is defined as the $\IZ$-algebra obtained as the quotient of $\Xi$ modulo the relations $y^\gamma(s,t)=0$ for all $s,t\in S$ and all $\gamma\in\mathbb{Z}$.

By abuse of notation the quotient map $\Xi\to\Omega$ will not be explicitly mentioned for the remainder of this paper and symbols like $e_s$, $x_s$ and $\iota(T_s)$ will therefore be used for elements of $\Xi$ as well as the corresponding elements of $\Omega$.
\end{definition}

The definition, and in particular the observation $x_s^2=(e_s x_s)(e_s x_s) = 0$, immediately implies that $T_s\mapsto\iota(T_s)$ defines a homomorphism of $\IZ[v^{\pm1}]$-algebras $\iota: H\to\mathbb{Z}[v^{\pm1}]\Omega$ (which is in fact injective as we will prove in corollary \ref{corollary:H_Omega_injective}). This observation also appears in Gyoja's paper \citep[remark 2.4.3]{Gyoja}.

\subsection{Morphisms}

Giving an algebra by generators and relations means having a universal property for homomorphisms on the resulting algebra. Since the relations for $\Omega$ are not explicit enough to be verifiable by explicit calculations we will use the following universal property instead.
\begin{lemma}
Consider the category of all rings. Then pre-composing with the quotient $\Xi\to\Omega$ is a natural isomorphism
\[\Hom(\Omega, A) \isomorphic \Big\lbrace f:\Xi\to A \mid \begin{array}{l}
\text{The induced map }\IZ[v^{\pm1}]\Xi \to \IZ[v^{\pm1}]A \\
\text{annihilates } \Delta_{m_{st}}(\iota(T_s),\iota(T_t)) \text{ for all }s,t\in S\end{array} \Big\rbrace.\]
\end{lemma}
\begin{proof}
Pre-composing with the quotient map certainly is an injective natural transformation $\Hom(\Omega,-)\to\Hom(\Xi,-)$. We will prove that its image is exactly the subset of the claim.

Choose $s,t\in S$ and write $\Delta_{m_{st}}(\iota(T_s),\iota(T_t)) = \sum_{\gamma\in\IZ} y^\gamma(s,t) v^\gamma$ as before. Thus for any homomorphism $f: \Xi\to A$ the induced map $\IZ[v^{\pm1}]\Xi \to \IZ[v^{\pm1}]A$ satisfies
\[f(\Delta_{m_{st}}(\iota(T_s),\iota(T_t))) = \sum_{\gamma\in\IZ} f(y^\gamma(s,t)) v^\gamma.\]
Because an element $\sum_{\gamma} a_\gamma v^\gamma \in \IZ[v^{\pm1}]A$ with $a_\gamma\in A$ is zero if and only if $a_\gamma=0$ for all $\gamma\in\IZ$ the map $f$ descends to a well-defined homomorphism $\Omega\to A$ if and only if $f$ annihilates all $y^\gamma(s,t)$ if and only if the induced map annihilates all braid commutators $\Delta_{m_{st}}(\iota(T_s),\iota(T_t))$.
\end{proof}

The following easy corollary establishes symmetries of $\Omega$ which will be used to simplify the proofs of the decomposition conjecture in the last section of the paper.

\begin{corollary}

\begin{enumerate}
	\item If $\alpha:S\to S$ is a bijection with $\ord(\alpha(s)\alpha(t))=\ord(st)$ (in other words a graph automorphism of the Dynkin diagram of $(W,S)$) then there is a unique automorphism of $\Omega$ with $e_s\mapsto e_{\alpha(s)}$, $x_s\mapsto x_{\alpha(s)}$.
	\item There is a unique antiautomorphism $\delta$ of $\Omega$ with $e_s\mapsto 1-e_s$, $x_s\mapsto -x_s$.
\end{enumerate}
\end{corollary}

\subsection{Modules and $W$-graphs}

The following definition appears in Gyoja's paper (\citep[definition 2.5]{Gyoja}) although with different notation:
\begin{definition}
In $\Xi$ define the following elements for all $I,J\subseteq S, s\in S$:
\[E_I:=\Big(\prod_{t\in I}e_t\Big)\Big(\prod_{t\in S\setminus I}(1-e_t)\Big)\]
\[X_{IJ}^s:=E_I x_s E_J\]
\end{definition}

What Gyoja did not mention in his paper is that these elements actually give $\Omega$ the structure of a quotient of a path algebra. This is the content of the following lemma.
\begin{lemma}\label{lemma:Xi_path_alg}
With the above notation the following statements are true:
\begin{enumerate}
	\item $E_I E_J = \delta_{IJ} E_I$, $\sum_{I\subseteq S} E_I = 1$ and $e_s = \sum_{\substack{I\subseteq S \\ s\in I}} E_I$.
	\item $X_{IJ}^s = 0$ if $s\notin I\setminus J$ and $x_s = \sum_{\substack{I,J\subseteq S \\ s\in I\setminus J}} X_{IJ}^s$.
	\item $\Xi$ is isomorphic to the path algebra $\IZ\mathcal{Q}$ over the quiver $\mathcal{Q}$ whose vertex set is the power set of $S$ and that has exactly $|I\setminus J|$ edges $I\leftarrow J$ for every pair of vertices $I,J\subseteq S$.
\end{enumerate}
\end{lemma}
\begin{proof}
The first equation follows immediately from the definition, $e_s(1-e_s)=(1-e_s)e_s=0$ and the fact that the $e_s$ commute with each other. The decomposition of the identity follows by expanding $1=\prod_{s\in S} (e_s+(1-e_s))$, and the expression for $e_s$ follows by applying the decomposition of the identity in $e_s\cdot 1$.

The expression for $x_s$ follows by applying the decomposition of the identity twice in $1\cdot x_s \cdot 1$.

The path algebra $\mathbb{Z}\mathcal{Q}$ can be described as the algebra freely generated by $\Set{\tilde{E}_K, \tilde{X}_{IJ}^s | K,I,J\subseteq S, s\in I\setminus J}$ with respect to the relations
\[\tilde{E}_I \tilde{E}_J = \delta_{IJ} \tilde{E}_I,\quad\sum_{I\subseteq S} \tilde{E}_I = 1\quad\textrm{and}\quad \tilde{X}_{IJ}^s = \tilde{E}_I \tilde{X}_{IJ}^s \tilde{E}_J.\]
This implies that $\tilde{E}_I\mapsto E_I$, $\tilde{X}_{IJ}^s\mapsto X_{IJ}^s$ induces a ring homomorphism $\mathbb{Z}\mathcal{Q}\to\Xi$. Going in the other direction, one readily verifies that the unique ring homomorphism $\Xi\to\mathbb{Z}\mathcal{Q}$  with $e_s\mapsto\sum_{\substack{I\subseteq S \\ s\in I}} \tilde{E}_I$ and $x_s\mapsto\sum_{\substack{I,J\subseteq S \\ s\in I\setminus J}} \tilde{X}_{IJ}^s$ is inverse to the first morphism.
\end{proof}

\begin{remark}
For later use we observe that
\begin{enumerate}
	\item The algebra automorphismen induced by a graph automorphism $\alpha$ maps $E_I\mapsto E_{\alpha(I)}$ and $X_{IJ}^s \mapsto X_{\alpha(I)\alpha(J)}^{\alpha(s)}$.
	\item The antiautomorphism $\delta$ maps $E_I\mapsto E_{I^c}$ and $X_{IJ}^s \mapsto -X_{J^c I^c}^s$ where $I^c$ denotes the complement of $I$ in $S$.
\end{enumerate}
\end{remark}

The following theorem also appears in Gyoja's paper as a remark without proof and establishes the connection between $\Omega$ and $W$-graphs.

\begin{theorem}[c.f. {\citep[remark 2.7]{Gyoja}}]\label{theorem:Omega_modules_W_graphs}
Let $k$ be a commutative ring.
There is a correspondence between $\Omega$-modules and $W$-graphs by to choice of a suitable basis. More precisely the following statements hold:

\begin{enumerate}
	\item (From $W$-graphs to $\Omega$-modules)
	
Let $(\mathfrak{C},\mathcal{I},m)$ be a $W$-graph with edge weights in $k$. Define $\omega: k\Omega\to k^{\mathfrak{C}\times\mathfrak{C}}$ by
\[\omega(e_s)_{xy} := \begin{cases} 1 & x=y, s\in \mathcal{I}(x) \\ 0 & \text{otherwise} \end{cases} \quad\textrm{and}\quad \omega(x_s) := m^s.\]
Then $\omega$ is a well-defined $k$-algebra homomorphism such that the composition
\[k[v^{\pm1}]H \xrightarrow{\iota} k[v^{\pm1}]\Omega \xrightarrow{\omega} k[v^{\pm1}]^{\mathfrak{C}\times\mathfrak{C}}\]
is exactly the matrix representation of $H$ attached to $(\mathfrak{C},\mathcal{I},m)$.

	\item (From $\Omega$-modules to $W$-graphs)

Let $V$ be a $k\Omega$-module with representation $\omega: k\Omega\to\operatorname{End}_k(V)$. Define $V_I:=E_I V$ for all $I\subseteq S$.

If $V_I$ is a finitely generated free $k$-module and $\mathfrak{C}_I\subseteq V_I$ is a $k$-basis for all $I\subseteq S$, define $(\mathfrak{C},\mathcal{I},m)$ as follows: Set $\mathfrak{C}:=\bigcup_{I\subseteq S} \mathfrak{C}_I$, set $\mathcal{I}(x):=I$ for all $x\in\mathfrak{C}_I$ and define $m^s$ to be the matrix of $\omega(x_s)$ with respect to the basis $\mathfrak{C}$. With these definitions $(\mathfrak{C},\mathcal{I},m)$ is a $W$-graph and its $W$-graph module is $k[v^{\pm1}]\otimes_k V$.
\end{enumerate}
\end{theorem}
\begin{proof}
1. The matrices $\omega(e_s)$ and $\omega(x_s)$ satisfy the relations of $\Xi$ by definition of $W$-graphs. We will therefore view $\omega$ as a algebra homomorphism $\Xi\to k^{\mathfrak{C}\times\mathfrak{C}}$. Because $\omega(\iota(T_s))$ is exactly equal to the matrices $\omega(T_s)$ in the definition of $W$-graphs and those matrices satisfy the braid relations, it follows that $\omega$ descends to a homomorphism $\Omega\to k^{\mathfrak{C}\times\mathfrak{C}}$ by the universal property.

2. The second assertion is easily verified: The condition $m_{xy}^s \neq 0 \implies s\in \mathcal{I}(x)\setminus \mathcal{I}(y)$ follows from $X_{IJ}^s \neq 0 \implies s\in I\setminus J$. The matrices occurring in the definition of $W$-graphs are exactly the matrices $\omega(\iota(T_s))$ and hence satisfy the necessary braid relations because the elements $\iota(T_s)\in\Omega$ satisfies them.
\end{proof}

\begin{corollary}\label{corollary:H_Omega_injective}
If $W$ is finite, then the following hold:
\begin{enumerate}
	\item $\iota: k[v^{\pm 1}]H\to k[v^{\pm 1}]\Omega$ is injective.
	\item All $E_I$ are non-zero as elements of $k\Omega$.
\end{enumerate}
\end{corollary}

In particular $H$ will be considered as a subalgebra of the scalar extension $\IZ[v^{\pm1}]\Omega$ for the rest of this paper.

\begin{proof}
Consider the Kazhdan-Lusztig-$W$-graph as defined in \cite{KL}. It is a $W$-graph $(\mathfrak{C},\mathcal{I},m)$  with $\mathfrak{C}:=W$, $\mathcal{I}(w):=\left\lbrace s\in S \mid sw<w \right\rbrace$ and integer edge weights such that the associated $W$-graph module is the regular $H$-module. This can be considered as a $W$-graph with edge weights in $k$.

The representation $k[v^{\pm 1}]H \xrightarrow{\iota} k[v^{\pm 1}]\Omega \to k[v^{\pm 1}]^{W\times W}$ induced by this $W$-graph equals the map $k[v^{\pm 1}]H\to\operatorname{End}_{k[v^{\pm 1}]}(k[v^{\pm 1}]H), h\mapsto(x\mapsto hx)$. The latter map is injective so that $\iota: k[v^{\pm 1}]H\to k[v^{\pm 1}]\Omega$ is injective too.

If $W$ is finite, then all the elements $E_I\in k\Omega$ are nonzero because there are $w\in\mathfrak{C}$ with $\mathcal{I}(w)=I$ (for example the longest elements of the corresponding parabolic subgroup $W_I$).
\end{proof}

\begin{remark}
The finiteness condition is in fact superfluous. A more carefully phrased version of the definition of $W$-graphs and of theorem \ref{theorem:Omega_modules_W_graphs} that also includes the infinite-dimensional case makes the same proof work for the first statement. The second statement however cannot be proved in the same way because there is an element $w\in W$ with $\mathcal{I}(w)=I$ if and only if $W_I$ is finite so that this proof doesn't work for infinite Coxeter groups (contrary to what I believed when I wrote my thesis which contains the special proof for the general statement). An alternative general proof of the second statement will be contained in my next paper \cite{Hahn2014canonical}.
\end{remark}

\section{$\Omega$ as a quotient of a path algebra}\label{section:relations}

It was already observed in Lemma \ref{lemma:Xi_path_alg} that $\Xi$ is a path algebra. In this section we will give an explicit set of relations for the quotient $\Xi\to\Omega$ in terms of this path algebra structure. The proof is inspired by equations appearing in Stembridge's paper \cite{Stembridge2008admissble}.

We will need the following lemma which is a slight generalization of \citep[Prop.\,3.1]{Stembridge2008admissble}.
\begin{lemma}\label{wgraph_alg:braid_commutator}
Define polynomials $\tau_r\in\IZ[T]$ by the following recursion:
\[\tau_{-1} :=0,\quad \tau_0 := 1, \quad \tau_r := T\tau_{r-1}-\tau_{r-2}\]
With this notation the following holds:

If $R$ is any ring and $x, y\in R$ are solutions of the equation $T^2 = 1+\zeta T$ for some fixed $\zeta\in R$, then their braid commutators satisfy
\[\Delta_{r+1}(x,y) = (-1)^r\tau_r(x+y-\zeta)\cdot(x-y).\]
\end{lemma}

Observe that $\tau_r$ is a monic polynomial of degree $r$ for all $r\in\IN$. In particular $\left\lbrace{\tau_0,\ldots,\tau_r}\right\rbrace$ is a $\IZ$-basis of $\left\lbrace{f\in\IZ[T] \mid \deg(f)\leq r}\right\rbrace$. Furthermore $\tau_r$ is an even polynomial for even $r$ and an odd polynomial for odd $r$, i.e. $\tau_r(-T) = (-1)^r \tau_r(T)$. This follows immediately from the recursion.

\begin{proof}
The claim for the braid commutator is true for $r=-1$ and $r=0$. Furthermore the following holds:
\begin{align*}
	(x+y)\Delta_{r+1}(x,y) &= x^2 \underbrace{y x \ldots}_{r} - \underbrace{x y x \ldots}_{r+2} + \underbrace{y x y\ldots}_{r+2} - y^2 \underbrace{x y\ldots}_{r}  \\
	&= (1+\zeta x) \underbrace{y x \ldots}_{r} - \underbrace{x y x \ldots}_{r+2} + \underbrace{y x y\ldots}_{r+2} - (1+\zeta y) \underbrace{x y\ldots}_{r}  \\
	&= 1\cdot\underbrace{y x \ldots}_{r} - 1\cdot\underbrace{x y \ldots}_{r} + \zeta x \underbrace{y x \ldots}_{r} - \zeta y \underbrace{x y \ldots}_{r} \\
	&\hspace{1em} -\big(\underbrace{x y x \ldots}_{r+2} - \underbrace{y x y\ldots}_{r+2}\big) \\
	&= -\Delta_r(x,y) + \zeta\Delta_{r+1}(x,y) - \Delta_{r+2}(x,y) \\
	\implies \Delta_{r+2}(x,y) &= (-1)\big((x+y-\zeta)\Delta_{r+1}(x,y) + \Delta_r(x,y)\big)
\end{align*}
The claim follows by induction.
\end{proof}

\begin{theorem}\label{theorem:Omega_relations}
For all $I,J\subseteq S$, $s,t\in S$ and $r\in\IN$ define
\[P_{IJ}^r(s,t) :=  E_I \underbrace{x_s x_t x_s\ldots}_{r\,\textrm{factors}} E_J = \begin{cases} 0 & r=0, I\neq J \\ E_I & r=0, I=J \\ \sum\limits_{I_1,\ldots,I_{r-1}\subseteq S} X_{I I_1}^s X_{I_1 I_2}^t X_{I_2 I_3}^s \ldots X_{I_{r-1} J}^s & r>0, 2\nmid r \\ \sum\limits_{I_1,\ldots,I_{r-1}\subseteq S} X_{I I_1}^s X_{I_1 I_2}^t X_{I_2 I_3}^s \ldots X_{I_{r-1} J}^t & r>0, 2\mid r \end{cases}.\]

With this notation the kernel of the quotient $\Xi\to\Omega$ is generated by the following elements:
\begin{enumerate}
	\item[$(\alpha)$] For all $s,t\in S$ the elements
	\[P_{IJ}^{m-1}(s,t) + a_{m-2} P_{IJ}^{m-2}(s,t) + \ldots + a_1 P_{IJ}^1(s,t) + a_0 P_{IJ}^0(s,t)\]
	for all $I,J\subseteq S$ where either
	\begin{itemize}
		\item $s\in I$, $t\notin I$, $s\in J$, $t\notin J$ and $2\nmid m_{st}$ or
		\item $s\in I$, $t\notin I$, $s\notin J$, $t\in J$ and $2\mid m_{st}$
	\end{itemize}  holds. The $a_i$ denote the coefficients of the polynomial $\tau_{m-1}$, i.e.
	\[\tau_{m-1}(T) = T^{m-1} + a_{m-2} T^{m-2} + \ldots + a_1 T + a_0.\]
	\item[$(\beta)$] For all $s,t\in S$ and all $I,J\subseteq S$ with $s,t\in I\setminus J$ the elements
	\[P_{IJ}^1(s,t) - P_{IJ}^1(t,s), P_{IJ}^2(s,t) - P_{IJ}^2(t,s), \ldots, P_{IJ}^m(s,t) - P_{IJ}^m(t,s).\]
\end{enumerate}
\end{theorem}

These relations will be used throughout the rest of the paper. We will refer to them as the $(\alpha^{st})$-relation and $(\beta^{st})$-relations respectively.

\begin{proof}
Consider $V:=\IZ[v^{\pm1}]\Xi$ and fixed $s,t\in S$. Define the four subspaces
\[V_{00} := \bigoplus_{\substack{I\subseteq S \\ s\notin I,t\notin I}} \!\!V E_I,\quad
  V_{01} := \bigoplus_{\substack{I\subseteq S \\ s\in I,t\notin I}} \!\!V E_I,\quad
  V_{10} := \bigoplus_{\substack{I\subseteq S \\ s\notin I,t\in I}} \!\!V E_I,\quad
  V_{11} := \bigoplus_{\substack{I\subseteq S \\ s\in I, t\in I}} \!\!V E_I.\]

Note that given an algebra $A$ and a decomposition into pairwise orthogonal idempotents $1_A=\sum_{i=1}^n e_i$ every element $a\in A$ can be uniquely written as $a=\sum_{i,j} a_{ij}$ with $a_{ij}\in e_i A e_j$ and this additive decomposition behaves like matrices behave with respect to multiplication, i.e. $(ab)_{ik} = \sum_j a_{ij} b_{jk}$.

We will therefore write elements of $\IZ[v^{\pm 1}]\Xi$ as matrices when we want to display such a decomposition in an efficient way. Note that one can view these matrices equivalently either as $d\times d$-matrices with entries in the Laurent polynomial ring $\IZ[v^{\pm1}]\Xi$ and as Laurent polynomials over the matrix ring $\Xi^{d\times d}$. In other words $\IZ[v^{\pm1}]\otimes (\Xi^{d\times d}) = (\IZ[v^{\pm 1}]\otimes \Xi)^{d\times d}$. It is therefore sensible to speak of the coefficient of $v^k$ of a matrix.

\medbreak
The matrices of $\iota(T_s)=-v^{-1}e_s + x_s + v(1-e_s)$ and $\iota(T_t)$ are given by
\begin{align*}
\iota(T_s) &=
-v^{-1}\begin{pmatrix}0&0&0&0\\0&1&0&0\\0&0&0&0\\0&0&0&1\end{pmatrix}
+\begin{pmatrix}0&0&0&0\\B_1&0&A_1&0\\0&0&0&0\\D_1&0&C_1&0\end{pmatrix}
+v\begin{pmatrix}1&0&0&0\\0&0&0&0\\0&0&1&0\\0&0&0&0\end{pmatrix} \\
&=\begin{pmatrix}
v & 0 & 0 & 0 \\
B_1 & -v^{-1} & A_1 & 0 \\
0 & 0 & v & 0 \\
D_1 & 0 & C_1 & -v^{-1}
\end{pmatrix}
\intertext{and}
\iota(T_t)&=
-v^{-1}\begin{pmatrix}0&0&0&0\\0&0&0&0\\0&0&1&0\\0&0&0&1\end{pmatrix}
+\begin{pmatrix}0&0&0&0\\0&0&0&0\\B_2&A_2&0&0\\D_2&C_2&0&0\end{pmatrix}
+v\begin{pmatrix}1&0&0&0\\0&1&0&0\\0&0&0&0\\0&0&0&0\end{pmatrix} \\
&=\begin{pmatrix}
v & 0 & 0 & 0 \\
0 & v & 0 & 0 \\
B_2 & A_2 & -v^{-1} & 0 \\
D_2 & C_2 & 0 & -v^{-1}
\end{pmatrix}
\end{align*}
respectively, where
\[A_1 = \sum_{\substack{I,J\subseteq S \\ s\in I, s\notin J \\ t\notin I, t\in J}} X_{IJ}^s\quad\textrm{and}\quad
  A_2 = \sum_{\substack{I,J\subseteq S \\ s\notin I, s\in J \\ t\in I, t\notin J}} X_{IJ}^t,\]
\[B_1 = \sum_{\substack{I,J\subseteq S \\ s\in I, s\notin J \\ t\notin I, t\notin J}} X_{IJ}^s\quad\textrm{and}\quad
  B_2 = \sum_{\substack{I,J\subseteq S \\ s\notin I, s\notin J \\ t\in I, t\notin J}} X_{IJ}^t,\]
\[C_1 = \sum_{\substack{I,J\subseteq S \\ s\in I, s\notin J \\ t\in I, t\in J}} X_{IJ}^s\quad\textrm{and}\quad
  C_2 = \sum_{\substack{I,J\subseteq S \\ s\in I, s\in J \\ t\in I, t\notin J}} X_{IJ}^t\quad\textrm{as well as}\]
\[D_1 = \sum_{\substack{I,J\subseteq S \\ s\in I, s\notin J \\ t\in I, t\notin J}} X_{IJ}^s\quad\textrm{and}\quad
  D_2 = \sum_{\substack{I,J\subseteq S \\ s\in I, s\notin J \\ t\in I, t\notin J}} X_{IJ}^t.\]
Finally define $z$ to be $v+v^{-1}$.

\medbreak
Step 1: We claim that for all $r\in\IN$
\begin{equation}
\Delta_{r+1}(\iota(T_s),\iota(T_t)) = (-1)^r\begin{pmatrix}
0 & 0 & 0 \\
\tau_r(A) JB & \tau_r(A)J(A-z) & 0 \\
X_r &  -C\tau_r(A)J & 0 
\end{pmatrix}
\label{eq:wgraph_alg:braid_commutator2}\tag{$\ast$}
\end{equation}
holds where
\[A:=\begin{pmatrix} 0 & A_1 \\ A_2 & 0 \end{pmatrix},\enspace B:=\begin{pmatrix} B_1 \\ B_2 \end{pmatrix},\enspace C:=\begin{pmatrix} C_2 & C_1 \end{pmatrix},\enspace J:=\begin{pmatrix} 1 & 0 \\ 0 & -1 \end{pmatrix} \enspace\textrm{and}\]
\[X_r := \sum_{i=0}^{r-1} (-1)^iC\tau_i(z)\tau_{r-1-i}(A)JB+(-1)^r\tau_r(z)(D_1-D_2).\]

In order to prove this claim define
\begin{align*}
	E &:= \iota(T_s)+\iota(T_t)-(v-v^{-1}) = \begin{pmatrix}
		z & 0 & 0 \\
		B & A & 0 \\
		D_1+D_2 & C & -z
	\end{pmatrix} \quad\textrm{and} \\
	F &:= \iota(T_s)-\iota(T_t) = \begin{pmatrix}
		0 & 0 & 0 \\
		JB & J(A-z) & 0 \\
		D_1-D_2 & -CJ & 0 
	\end{pmatrix}.
\end{align*}
By Lemma \ref{wgraph_alg:braid_commutator} $\Delta_{r+1}(\iota(T_s),\iota(T_t))=(-1)^r \tau_r(E)F$. Therefore we will inductively show that $\tau_r(E)F$ equals the matrix in \eqref{eq:wgraph_alg:braid_commutator2}. For $r=-1$ and $r=0$ this is clear. The induction step follows from
\begin{align*}
	\tau_{r+1}(E)F &= E\tau_r(E)F - \tau_{r-1}(E)F \\
	&= \begin{pmatrix}
	z & 0 & 0 \\ B & A & 0 \\ D_1+D_2 & C & -z
	\end{pmatrix}\cdot\begin{pmatrix}
	0&0&0 \\ \tau_r(A)JB & \tau_r(A)J(A-z) & 0 \\ X_r & -C\tau_r(A)J & 0
	\end{pmatrix} \\
	&\phantom{\textrm{= }} - \begin{pmatrix}
	0&0&0 \\ \tau_{r-1}(A)JB & \tau_{r-1}(A)J(A-z) & 0 \\ X_{r-1} & -C\tau_{r-1}(A)J & 0
	\end{pmatrix} \\
	&=\begin{pmatrix}
	0 & 0 & 0 \\ A\tau_r(A)JB-\tau_{r-1}(A)JB & H & 0 \\
	L & K & 0
	\end{pmatrix}
\end{align*}
where we used the abbreviations
\begin{align*}
	H &:= A\tau_r(A)J(A-z)-\tau_{r-1}(A)J(A-z) \quad\textrm{and} \\
	K &:= C\tau_r(A)J(A-z)+zC\tau_r(A)J+C\tau_{r-1}(A)J \\
	L &:= C\tau_r(A)JB-zX_r-X_{r-1}
\end{align*}
At the positions $(2,1)$ and $(2,2)$ the term is clearly equal to the desired result. At position $(3,2)$ we use $JA=-AJ$ and simplify the expression as follows
\begin{align*}
	K &= C\tau_r(A)JA-C\tau_r(A)Jz+zC\tau_r(A)J+C\tau_{r-1}(A)J \\
	&= -C\tau_r(A)AJ+C\tau_{r-1}(A)J \\
	&= -C\tau_{r+1}(A)J.
\end{align*}
Using the recursive definition of $\tau_{r+1}$ it is also a routine calculation to show that 
\begin{align*}
	L &= \sum_{i=0}^r (-1)^i C\tau_i(z)\tau_{r-i}(A)JB+(-1)^{r+1}\tau_{r+1}(z)(D_1-D_2)
\end{align*}

This shows \eqref{eq:wgraph_alg:braid_commutator2}.

\bigbreak
Step 2: Simplify the result

Now let $\mathfrak{K}=\ker(\Xi\to\Omega)$. By definition this ideal is generated by the coefficients of the $v^\gamma$ in $\Delta_m(\iota(T_s),\iota(T_t))\in\IZ[v^{\pm 1}]\Xi$. Therefore we have consider the coefficients of
\begin{enumerate}
	\item $R_1:=\tau_{m-1}(A)JB$,
	\item $R_2:=\tau_{m-1}(A)J(A-z)$,
	\item $R_3:=\sum_{i=0}^{m-2} (-1)^iC\tau_i(z)\tau_{m-2-i}(A)JB+(-1)^{m-1}\tau_{m-1}(z)(D_1-D_2)$ and
	\item $R_4:=C\tau_{m-1}(A)J$.
\end{enumerate}
The coefficient of the highest power of $v$ in $R_2$ is $-\tau_{m-1}(A)J$ because $z=v+v^{-1}$ so that the coefficient of the highest power of $z$ is also the coefficient of the highest power of $v$ in any Laurent polynomial. Now $R_2$ is contained in $\mathfrak{K}[v^{\pm1}]$  (remember that we view these matrices as elements of $\IZ[v^{\pm1}]\Xi$ so that this makes sense) if and only if $\tau_{m-1}(A)\in\mathfrak{K}$ because $J$ is invertible. Conversely $R_1$, $R_2$ and $R_4$ are in $\mathfrak{K}[v^{\pm 1}]$ if $\tau_{m-1}(A)\in\mathfrak{K}$ holds.

Let's have a closer look at $R_3$: The polynomial $\tau_r$ has degree $r$. The coefficient of the highest power of $v$ in $R_3$ equals $(-1)^{m-1}(D_1-D_2)$. Therefore $D_1-D_2\in\mathfrak{K}$ and $R_3$ is in $\mathfrak{K}[v^{\pm 1}]$ if and only if $D_1-D_2\in\mathfrak{K}$ and $R_3'=\sum_{i=0}^{m-2} (-1)^iC\tau_i(z)\tau_{r-2-i}(A)JB\in\mathfrak{K}[v^{\pm 1}]$. Looking repeatedly at the coefficient of the highest power of $v$ and shortening the term we get that $R_3'$ is in $\mathfrak{K}[v^{\pm 1}]$ if and only if $C\tau_0(A) JB, C\tau_1(A)JB, \ldots, C\tau_{m-2}(A)JB\in\mathfrak{K}$. Because $\Set{\tau_0,\ldots,\tau_{m-2}}$ is a $\IZ$-basis of $\lbrace f\in\IZ[T] \mid \deg(f)\leq m-2 \rbrace$ these terms are in $\mathfrak{K}$ if and only if $CA^0 JB, CA^1 JB, \ldots, CA^{m-2} JB$ are.

Thus we obtain the generating set
\begin{enumerate}
	\item[$(\alpha)$] $R_\alpha:=\tau_{m-1}(A)$,
	\item[$(\beta)$]  $R_\beta:=D_1-D_2$ and
	\item[$(\gamma)$] $R_{\gamma,k}:=CA^k JB$ for $0\leq k \leq m-2$
\end{enumerate}
for the ideal $\mathfrak{K}$.

\bigbreak
Step 3: The relations.

Again we decompose $\Xi$ as $\bigoplus_I \Xi E_I$ and use $R\in\mathfrak{K}$ if and only if $E_I R E_J$ in $\mathfrak{K}$ for all $I,J\subseteq S$.

\medbreak
To determine $E_I R_\alpha E_J$ we consider $E_I A^k E_J$. For $k=0$ this simplifies to $E_I A^0 E_J=\delta_{IJ} E_I=P_{IJ}^0$. For $k>0$ we obtain
\[A^k = \begin{cases}
\begin{pmatrix}	(A_1 A_2)^{\frac{k}{2}} & 0 \\ 0 & (A_2 A_1)^{\frac{k}{2}} \end{pmatrix} & \textrm{if}\;\;2\mid k \\
\begin{pmatrix}	0 & (A_1 A_2)^{\frac{k-1}{2}} A_1 \\ (A_2 A_1)^{\frac{k-1}{2}} A_2 & 0 \end{pmatrix} & \textrm{if}\;\;2\nmid k
\end{cases}\]
and substitute
\[A_1 = \sum_{\substack{I,J\subseteq S \\ s\in I, s\notin J \\ t\notin I, t\in J}} X_{IJ}^s,\quad
  A_2 = \sum_{\substack{I,J\subseteq S \\ s\notin I, s\in J \\ t\in I, t\notin J}} X_{IJ}^t\]
to obtain
\[\sum_{I_0,I_1,\ldots,I_k\subseteq S} X_{I_0 I_1}^s X_{I_1 I_2}^t X_{I_2 I_3}^s \ldots\]
where the sum is over all $I_i$ that satisfy $s\in I_{2i}\setminus I_{2i+1}$ and $t\in I_{2i+1}\setminus I_{2i}$ when we consider $A_1 A_2 A_1 \ldots$. Because $X_{IJ}^s = 0$ if $s\notin I\setminus J$, only the conditions for $I=I_0$ and $I_k=J$ are not vacuous. Therefore we could just sum over all paths of length $k$ that go to $I$ from $J$. We therefore obtain
\[\underbrace{A_1 A_2 \ldots}_{k} = \begin{cases}
	\sum\limits_{\substack{I,J\subseteq S \\ s\in I, s\in J \\ t\notin I, t\notin J}} P_{IJ}^k(s,t) & \textrm{if}\;\;2\mid k \\
	\sum\limits_{\substack{I,J\subseteq S \\ s\in I, s\notin J \\ t\notin I, t\in J}} P_{IJ}^k(s,t) & \textrm{if}\;\;2\nmid k
	\end{cases}.\]
For the other product we similarly obtain
\[\underbrace{A_2 A_1 \ldots}_{k} = \begin{cases}
	\sum\limits_{\substack{I,J\subseteq S \\ s\notin I, s\notin J \\ t\in I, t\in J}} P_{IJ}^k(t,s) & \textrm{if}\;\;2\mid k \\
	\sum\limits_{\substack{I,J\subseteq S \\ s\notin I, s\in J \\ t\in I, t\notin J}} P_{IJ}^k(t,s) & \textrm{if}\;\;2\nmid k
	\end{cases}\]
Multiplying with $E_I$ from the left and with $E_J$ from the right this equals either $0$ or $P_{IJ}^k(s,t)$ and $P_{IJ}^k(t,s)$ respectively. The element $E_I \tau_{m-1}(A) E_J\in\mathfrak{K}$ is, if it is not zero, equal to
\[P_{IJ}^{m-1}(s,t) + a_{m-2} P_{IJ}^{m-2}(s,t) + \ldots + a_2 P_{IJ}^2(s,t) + a_1 P_{IJ}^1(s,t) + a_0 P_{IJ}^0(s,t)\]
where $\tau_{m-1}(T) = T^{m-1} + a_{m-2} T^{m-2} + \ldots + a_2 T^2 + a_1 T^1 + a_0$ and similarly for the symmetric situation where $s,t$ are swapped.

\medbreak
The second kind of generators is easier: $R_\beta$ is equal to
\[\sum_{\substack{I,J\subseteq S \\ s\in I, s\notin J \\ t\in I, t\notin J}} X_{IJ}^s - X_{IJ}^t.\]
For those $I$, $J$ that do not occur in this sum $E_I R_\beta E_J=0$. For all others we obtain the element $X_{IJ}^s - X_{IJ}^t=P_{IJ}^1(s,t) - P_{IJ}^1(t,s)$. This is first case in the relations of type $(\beta)$.

\medbreak
Finally there is only one kind of generators left, namely $R_{\gamma,k} = CA^k JB$. We already know the powers of $A$ and therefore obtain
\[CA^kJB = \begin{cases}
C_2 (A_1 A_2)^{\frac{k}{2}} B_1 - C_1 (A_2 A_1)^{\frac{k}{2}} B_2 & \textrm{if}\;\;2\mid k \\
- C_2 (A_1 A_2)^{\frac{k-1}{2}} A_1 B_2 + C_1 (A_2 A_1)^{\frac{k-1}{2}} A_2 B_1 & \textrm{if}\;\;2\nmid k
\end{cases}.\]
We substitute the definitions
\[B_1 = \sum_{\substack{I,J\subseteq S \\ s\in I, s\notin J \\ t\notin I, t\notin J}} X_{IJ}^s,\quad
  B_2 = \sum_{\substack{I,J\subseteq S \\ s\notin I, s\notin J \\ t\in I, t\notin J}} X_{IJ}^t\]
\[C_1 = \sum_{\substack{I,J\subseteq S \\ s\in I, s\notin J \\ t\in I, t\in J}} X_{IJ}^s,\quad
  C_2 = \sum_{\substack{I,J\subseteq S \\ s\in I, s\in J \\ t\in I, t\notin J}} X_{IJ}^t.\]
If $s,t\in I$ or $s,t\notin J$ is not satisfied, then $E_I CA^k JB E_J=0$ because either $E_I C=0$ or $BE_J=0$. Otherwise
\begin{align*}
	E_I CA^kJB E_J &= \sum_{I_0,\ldots,I_k\subseteq S} X_{I I_0}^t X_{I_0 I_1}^s \ldots X_{I_{k-1} I_k}^t X_{I_k J}^s - X_{I, I_0}^s X_{I_0 I_1}^t \ldots X_{I_{k-1} I_k}^s X_{I_k J}^t \\
	&= P_{IJ}^{k+2}(t,s) - P_{IJ}^{k+2}(s,t)
\end{align*}
holds if $2\mid k$ and
\begin{align*}
	E_I CA^kJB E_J &= \sum_{I_0,\ldots,I_k\subseteq S} - X_{I I_0}^t X_{I_0 I_1}^s \ldots X_{I_{k-1} I_k}^s X_{I_k J}^t + X_{I I_0}^s X_{I_0 I_1}^t \ldots X_{I_{k-1} I_k}^t X_{I_k J}^s \\
	&= -P_{IJ}^{k+2}(t,s) + P_{IJ}^{k+2}(s,t)
\end{align*}
holds if $2\nmid k$. This provides the other elements in the relations of type $(\beta)$.
\end{proof}

Because of the $(\beta)$-relation $X_{IJ}^s - X_{IJ}^t = P_{IJ}^1(s,t) - P_{IJ}^1(t,s) = 0$ in $\Omega$, the upper index of these elements does not matter and it is well-defined to write $X_{IJ}$ for the common value of $X_{IJ}^s\in\Omega$ for all $s\in I\setminus J$. We will adopt this notation for the rest of this article.

Additionally $(\alpha)$ implies that $X_{IJ}^s = X_{JI}^t = 0$ holds in $\Omega$ for all $I,J\subseteq S$, $s\in I\setminus J$, $t\in J\setminus I$ with $m_{st}=2$ (i.e. $s$ and $t$ are not connected in the Dynkin-diagram of $(W,S)$). This allows us to think of $\Omega$ as a quotient of a path algebra over a much simpler quiver which was defined by Stembridge \citep[section 4]{Stembridge2008admissble}.
\begin{definition}
The \emph{compatibility graph} of $(W,S)$ is the directed graph $\mathcal{Q}_W$ with vertex set $\left\lbrace I \mid I\subseteq S\right\rbrace$ and a single edge $I\leftarrow J$ if and only if $I\setminus J\neq\emptyset$ and no element of $I\setminus J$ commutes with any element of $J\setminus I$.

An edge $I\leftarrow J$ with $I\supseteq J$ is called an \emph{inclusion edge}, all other edges are called \emph{transversal edges}.
\end{definition}

Note that transversal edges always occur in pairs of opposite orientation because their definition is symmetric: $I\leftarrow J$ is a transversal edge if and only if $I\setminus J\neq\emptyset$, $J\setminus I\neq\emptyset$ and all $s\in I\setminus J$ are connected to all $t\in J\setminus I$ in the Dynkin-diagram of $(W,S)$.

\begin{corollary}
$\Omega$ is a quotient of the path algebra $\mathbb{Z}\mathcal{Q}_W$.
\end{corollary}
\begin{proof}
Denote the vertex elements in $\IZ\mathcal{Q}_W$ by $\tilde{E}_I$ and the edge elements by $\tilde{X}_{IJ}$. Then $\tilde{E}_I\mapsto E_I$ and $\tilde{X}_{IJ} \mapsto X_{IJ}^s$ with any $s\in I\setminus J$ extends to a well-defined algebra homomorphism $\IZ\mathcal{Q}_W\to\Omega$ by the $(\beta)$-relation. It is surjective because all elements $X_{IJ}^s\in\Omega$ are either zero by some $(\alpha)$-type relation as seen above or contained in the image of this morphism and $\Omega$ is generated by the elements $X_{IJ}^s$ and $E_I$.
\end{proof}

\begin{example}
Figure \ref{fig:wgraph_alg:comp_graphs} displays the compatibility graphs of the finite irreducible Coxeter groups of rank $\leq 4$. For the sake of clarity inclusion edges are only displayed in rank 2 and rank 3 and only between those $I,J\subseteq S$ that satisfy $\abs{I\setminus J}=1$. Pairs of transversal edges $I\leftrightarrows J$ are combined into one (bold) undirected edge.

\begin{figure}[htp]
	\centering
	\begin{tabular}{c|c}
		\begin{minipage}[c][ 75pt]{ 75pt}
		\begin{tikzpicture}

\node[Vertex] (E_empty) at (0,0) {$\emptyset$};

\node[Vertex] (E_1) at (-1,1) {$1$};
\node[Vertex] (E_2) at (+1,1) {$2$};

\node[Vertex] (E_12) at (0,2) {$12$};

\path[EdgeI]
	(E_empty) edge (E_1)
	(E_empty) edge (E_2)
	(E_1) edge (E_12)
	(E_2) edge (E_12);

\path[EdgeT]
	(E_1) edge (E_2);
\end{tikzpicture}
		\end{minipage}	& 
		\begin{minipage}[c][100pt]{100pt}
		\begin{tikzpicture}

\node[Vertex] (E_empty) at (0,0) {$\emptyset$};

\node[Vertex] (E_1) at (-1.5,1) {$1$};
\node[Vertex] (E_2) at ( 0.0,1) {$2$};
\node[Vertex] (E_3) at (+1.5,1) {$3$};

\node[Vertex] (E_23) at (-1.5,2) {$23$};
\node[Vertex] (E_13) at ( 0.0,2) {$13$};
\node[Vertex] (E_12) at (+1.5,2) {$12$};

\node[Vertex] (E_123) at (0,3) {$123$};

\path[EdgeI]
	(E_empty) edge (E_1)
	(E_empty) edge (E_2)
	(E_empty) edge (E_3)
	(E_1) edge (E_12)
	(E_1) edge (E_13)
	(E_2) edge (E_12)
	(E_2) edge (E_23)
	(E_3) edge (E_13)
	(E_3) edge (E_23)
	(E_12) edge (E_123)
	(E_13) edge (E_123)
	(E_23) edge (E_123);

\path[EdgeT]
	(E_1) edge (E_2)
	(E_2) edge (E_3)
	      edge (E_13)
	(E_23) edge (E_13)
	(E_13) edge (E_12);

\end{tikzpicture}
		\end{minipage} \\
		\hline
		\begin{minipage}[c][320pt]{200pt}
		\begin{tikzpicture}

\node[Vertex] (E_empty) at (0,0) {$\emptyset$};

\node[Vertex] (E_1) at (-3.0,1) {$1$};
\node[Vertex] (E_2) at (-1.0,1) {$2$};
\node[Vertex] (E_3) at (+1.0,1) {$3$};
\node[Vertex] (E_4) at (+3.0,1) {$4$};

\node[Vertex] (E_12) at (-3.0,2.5) {$12$};
\node[Vertex] (E_13) at (-1.0,2.5) {$13$};
\node[Vertex] (E_23) at ( 0.0,2.0) {$23$};
\node[Vertex] (E_14) at ( 0.0,3.0) {$14$};
\node[Vertex] (E_24) at (+1.0,2.5) {$24$};
\node[Vertex] (E_34) at (+3.0,2.5) {$34$};

\node[Vertex] (E_123) at (-3.0,4) {$123$};
\node[Vertex] (E_124) at (-1.0,4) {$124$};
\node[Vertex] (E_134) at (+1.0,4) {$134$};
\node[Vertex] (E_234) at (+3.0,4) {$234$};

\node[Vertex] (E_1234) at (0,5) {$1234$};

%
%
%
%
%
%
%
%
%
%
%
%

\path[EdgeT]
	(E_1) edge (E_2)
	(E_2) edge (E_3) edge (E_13)
	(E_3) edge (E_4) edge (E_24)
	
	(E_13) edge (E_23) edge (E_14) edge (E_12)
	(E_24) edge (E_23) edge (E_14) edge (E_34)
	
	(E_123) edge (E_124)
	(E_124) edge (E_134) edge (E_13)
	(E_134) edge (E_234) edge (E_24);

\end{tikzpicture}
		\end{minipage} &
		\begin{minipage}[c][320pt]{150pt}
		\begin{tikzpicture}[
	Vertex/.style = {inner sep = 1pt,outer sep = 2pt,minimum size=15pt,
	                 circle,draw=black!70,thick,
	                 font=\scriptsize},
	EdgeT/.style = {ultra thick,black},
	EdgeI/.style = {black!80,->}
]

\node[Vertex] (E_empty) at (0,0) {$\emptyset$};

\node[Vertex] (E_0) at (-2,1) {$0$};
\node[Vertex] (E_1) at ( 0,1) {$1$};
\node[Vertex] (E_3) at (+2,1) {$3$};

\node[Vertex] (E_2) at ( 0,2) {$2$};

\node[Vertex] (E_13) at (-2,4) {$13$};
\node[Vertex] (E_03) at ( 0,4) {$03$};
\node[Vertex] (E_01) at (+2,4) {$01$};

\node[Vertex] (E_02) at (-2,6) {$02$};
\node[Vertex] (E_12) at ( 0,6) {$12$};
\node[Vertex] (E_23) at (+2,6) {$23$};

\node[Vertex] (E_013) at (0,8) {$013$};

\node[Vertex] (E_123) at (-2,9) {$123$};
\node[Vertex] (E_023) at ( 0,9) {$023$};
\node[Vertex] (E_012) at (+2,9) {$012$};

\node[Vertex] (E_1234) at (0,10) {$0123$};

%
%
%
%
%
%
%
%
%
%
%

\path[EdgeT]
	(E_2) edge (E_0)
	      edge (E_1)
	      edge (E_3)
	      
	      edge (E_13)
	      edge (E_03)
	      edge (E_01)
	(E_13) edge (E_02)
	       edge (E_12)
	(E_03) edge (E_02)
	       edge (E_23)
	(E_01) edge (E_12)
	       edge (E_23)

	(E_013) edge (E_02)
	        edge (E_12)
	        edge (E_23)
	        
	        edge (E_123)
	        edge (E_023)
	        edge (E_012)

	(E_2) edge[bend right=30] (E_013);

\end{tikzpicture} 
		\end{minipage}
	\end{tabular}
	
	\caption{Compatibility graphs for small Coxeter groups; top left for $I_2(m)$, top right for $A_3$, $B_3$ and $H_3$, bottom left for $A_4$, $B_4$ and $F_4$, bottom right for $D_4$.}
	\label{fig:wgraph_alg:comp_graphs}
\end{figure}
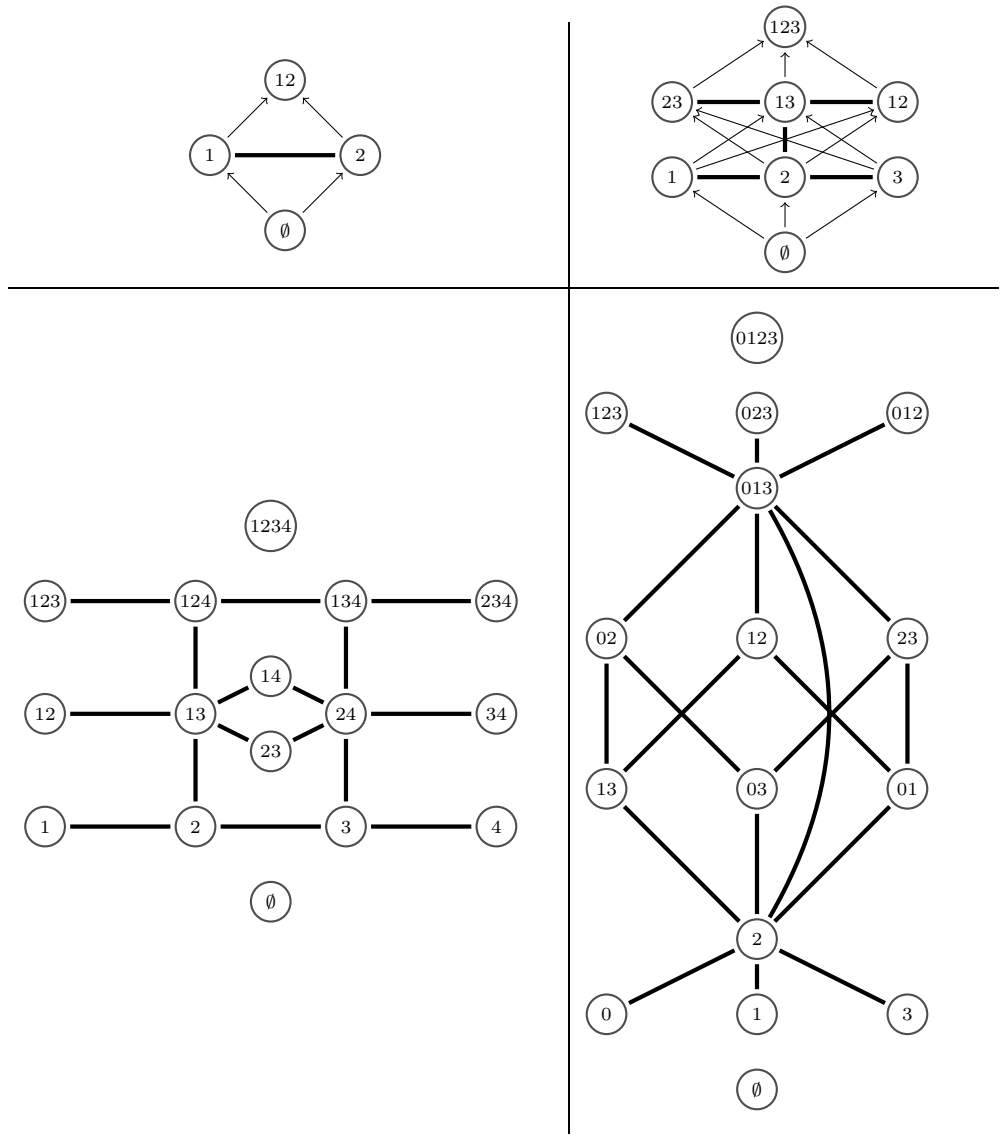
\end{example}

\section{The decomposition conjecture}\label{section:decomposition_conjecture}

While trying to prove Gyoja's conjecture\footnote{which states that the Jacobson radical of $K\Omega$ has codimension $|W|$ if $K$ is a sufficiently large field of characteristic zero or equivalently that two irreducible $K\Omega$-modules are isomorphic iff their restrictions to $K(v)H$ are isomorphic, c.f. \cite[remark 2.18]{Gyoja}, \citep[theorem 4.3.7]{hahn2013diss}.} and to better understand the internal structure of $\Omega$ and $W$-graphs I found a number of very similar proofs for some small types of Coxeter groups. The essence of these proofs is captured by the following four conjectural properties of $\Omega$.

\begin{conjecture}\label{conj:decomposition}
Let $k\subseteq\IC$ be a good ring for $(W,S)$. There exists a family $(F^\lambda)_{\lambda\in\Irr(W)}$ of elements of $k\Omega$ with the following properties:
\begin{enumerate}
	\item[(Z1)] The $F^\lambda$ are pairwise orthogonal idempotents and decompose the identity:
	\[\forall\lambda,\mu\in\Irr(W): F^\lambda F^\mu = \delta_{\lambda\mu} F^\lambda, \quad 1=\smash{\sum_{\lambda\in\Irr(W)}} F^\lambda\]
	\item[(Z2)] This decomposition is compatible with the decomposition induced by the path-algebra structure:
	\[\forall\lambda\in\Irr(W) \forall I\subseteq S: E_I F^\lambda = F^\lambda E_I\]
	\item[(Z3)] There is a partial order $\preceq$ on $\Irr(W)$ such that only ``downward edges'' exist: If $F^\lambda k\Omega F^\mu \neq 0$, then $\lambda \preceq \mu$.
	\item[(Z4)] There are surjective $k$-algebra morphisms $\psi_\lambda: k^{d_\lambda\times d_\lambda} \twoheadrightarrow F^\lambda k\Omega F^\lambda$ for all $\lambda\in\Irr(W)$ where $d_\lambda$ denotes the degree of the character $\lambda$.
\end{enumerate}
\end{conjecture}

\begin{remark}
The ``edge'' terminology in Z3 refers to the quiver $\Lambda$ in the next theorem.
\end{remark}

The decomposition conjecture is of interest because it implies several important properties of the $W$-graph algebra and its modules (that is: $W$-graphs) as the following theorem demonstrates.
\begin{theorem}
Assume that the decomposition conjecture is true for the finite Coxeter group $(W,S)$ and $k$ a good ring for $(W,S)$. Then the following properties holds:
\begin{enumerate}
	\item Consider the quiver $\Lambda$ which has $\Irr(W)$ as its set of vertices and an edge $\lambda\leftarrow\mu$ if and only if $\lambda\prec\mu$. Then $k\Omega$ is a quotient of the generalized path algebra (c.f. \cite{coelho2000pathalg}) over the quiver $\Lambda$ which has $k^{d_\lambda\times d_\lambda}$ as vertex algebras.
	\item $k\Omega$ is finitely generated as a $k$-module.
\end{enumerate}
Furthermore, if $k$ is a field, contain, then the following hold:
\begin{enumerate}
	\setcounter{enumi}{2}
	\item The Jacobson radical $\rad(k\Omega)$ is generated by the elements $F^\lambda X_{IJ} F^\mu$ with $\lambda\prec\mu$ and $k\Omega/\rad(k\Omega) \isomorphic \prod_{\lambda\in\Irr(W)} k^{d_\lambda\times d_\lambda}$.
	\item A $k\Omega$-module $V$ is simple if and only if the restriction of $k(v)V$ to $k(v)H$ is simple. Furthermore (after reindexing the family $(F^\lambda)_{\lambda\in\Irr(W)}$ if necessary) the latter has isomorphism class $\lambda$ if and only if $F^\lambda V=V$ holds.
	\item Every $k\Omega$-module $V$ has a family of natural submodules $(V^{\preceq\lambda})_{\lambda\in\Irr(W)}$ such that
	\begin{itemize}
		\item $\lambda\preceq\mu \implies V^{\preceq\lambda} \leq V^{\preceq\mu}$ and 
		\item $V^{\preceq\lambda} / V^{\prec\lambda}$ is isomorphic to a direct sum of irreducibles of isomorphism class $\lambda$ where $V^{\prec\lambda} := \sum_{\mu\prec\lambda} V^{\preceq\mu}$
	\end{itemize}
\end{enumerate}
\end{theorem}

\begin{remark}
Given that the Kazhdan-Lusztig-$W$-graph is indecomposable but not irreducible, it cannot be expected that an arbitrary $\Omega$-module decomposes as a direct sum of its irreducible constituents. The special filtration appearing in the above theorem is the next best thing one can hope for: One finds the irreducible constituents in the layers of a natural filtration and even nicely grouped into isomorphism classes. This fact and the first part of the theorem, saying that $\Omega$ itself is composed of much simpler parts like matrix algebras and path algebras, motivates the name ``decomposition conjecture''.

To the best of my knowledge, the decomposition conjecture has neither directly nor in a similar form been stated before in the literature, apart from my dissertation \cite{hahn2013diss}. The above consequences of conjecture \ref{conj:decomposition} also have not been considered or proved before, even in special cases, as far as I know.
\end{remark}

\begin{proof}
Denote with $k\widetilde{\Omega}$ this generalized path algebra and recall that it is characterized by the following properties:
\begin{itemize}
	\item It contains a set of pairwise orthogonal idempotents $f_\lambda$ corresponding to the vertices such that $\sum_\lambda f_\lambda = 1$ and $f_\lambda k\widetilde{\Omega} f_\lambda = k^{d_\lambda\times d_\lambda}$,
	\item It contains a set of elements $y_{\lambda\mu}$ corresponding to the edges $\lambda\leftarrow\mu$ such that $y_{\lambda\mu} = f_\lambda y_{\lambda\mu} f_\mu$,
	\item It satisfies the universal mapping property with respect to these features: For any $k$-algebra $A$, any set of elements $\Set{f'_\lambda, y'_{\lambda\mu} | \lambda,\mu\in\Irr(W)}$ satisfying these properties and any $\psi_\lambda:k^{d_\lambda\times d_\lambda}\to f'_\lambda A f'_\lambda$ there exists a unique morphism of $k$-algebras $\psi: k\widetilde{\Omega}\to A$ with $\psi(f_\lambda)=f'_\lambda, \psi(y_{\lambda\mu})= y'_{\lambda\mu}$ and $\psi{|f_\lambda A f_\lambda} = \psi_\lambda$.
\end{itemize}

We define elements $Y_{\lambda\mu} := \sum_{I,J\subseteq S} F^\lambda X_{IJ} F^\mu \in k\Omega$. Note that $Y_{\lambda\mu}=0$ if $\lambda\not\preceq\mu$ by Z3. The universal property ensures that the morphisms from Z4 $\psi_\lambda: k^{d_\lambda\times d_\lambda} \to k\Omega$ together with $f_\lambda\mapsto F^\lambda$ and $y_{\lambda\mu} \mapsto Y_{\lambda\mu}$ uniquely extend to an algebra morphism $\psi: k\widetilde{\Omega}\to k\Omega$ (this uses Z1).

We verify that this is an epimorpism. By construction
\[F_I^\lambda \overset{Z2}{=} F^\lambda E_I F^\lambda\in F^\lambda k\Omega F^\lambda \overset{Z4}{=} \im(\psi_\lambda) \subseteq \im(\psi)\]
for all $I\subseteq S,\lambda\in\Irr(W)$. Therefore $E_I=\sum_\lambda F_I^\lambda \in\im(\psi)$. Also $X_{IJ} \overset{Z3}{=} \sum_{\lambda} F^\lambda X_{IJ} F^\lambda + \sum_{\lambda\prec\mu} Y_{\lambda\mu}\in\im(\psi)$ for all $I,J\subseteq S$. Because we already know that $\{E_I, X_{IJ} \mid I,J\subseteq S\}$ generates $k\Omega$ we are done.

\medbreak
Because $\psi: k\widetilde{\Omega}\to k\Omega$ is surjective, $k\Omega$ is finitely generated as a $k$-module because $k\widetilde{\Omega}$ is, and $\psi(\rad(k\widetilde{\Omega}))\subseteq\rad(k\Omega)$ holds. The morphism $\psi$ therefore induces a surjection $k\widetilde{\Omega} / \rad(k\widetilde{\Omega}) \twoheadrightarrow k\Omega/\rad(k\Omega)$.

\medbreak
Now consider the case that $k$ is a field. The radical of the generalized path algebra is then easily seen to coincide with the ideal generated by the edge elements $y_{\lambda\mu}$ because the quiver $\Lambda$ is acyclic (see \citep[proposition 1.3]{coelho2000pathalg} for a more general characterisation of the radical of generalized path algebras). In fact $k\widetilde{\Omega} / \rad(k\widetilde{\Omega}) = \prod_{\lambda} k^{d_\lambda\times d_\lambda}$.

This implies $\dim_k k\Omega / \rad(k\Omega) \leq \sum_{\lambda} d_\lambda^2$. We show that $\prod_{\lambda} k^{d_\lambda\times d_\lambda}$ is in fact a quotient of $k\Omega$ to establish equality. For each $\lambda$ choose a $W$-graph with edge weights in $k$ realizing the irreducible $k(v)H$-module of isomorphism class $\lambda$ (this is possible by Gyoja's work \citep[theorem 2.3]{Gyoja}) and consider the induced $k\Omega$-module $V_\lambda$ (in particular $\dim_k V_\lambda=d_\lambda$), set $V:=\bigoplus_\lambda V_\lambda$ and denote the associated representation $k\Omega\to\operatorname{End}_k(V)$ by $\omega$. By construction $\im(\omega)\subseteq\prod_\lambda k^{d_\lambda\times d_\lambda}$ holds. Now consider $k(v)V$ as a module for $k(v)H \subseteq k(v)\Omega$ by restriction. Because $k(v)V$ contains each irreducible module of the Hecke algebra exactly once $\omega(k(v)H) = \prod_\lambda k(v)^{d_\lambda\times d_\lambda}$ holds. By comparing dimensions we obtain the desired equality above.

\medbreak
The fourth item follows from this. On one hand a $k\Omega$-module $V$ is certainly simple if its restriction to a subalgebra is already simple. Because every simple $k(v)H$-module can be realized by a $W$-graph, choosing one $W$-graph for each isomorphism class induces an injective map $\Irr(W)\isomorphic\Irr(k(v)H) \to \Irr(k\Omega)$ with the restriction map as an left inverse. Because of $k\Omega/\rad(k\Omega)\isomorphic \prod_\lambda k^{d_\lambda\times d_\lambda}$ the number of elements in both sets is the same so that the map is actually a bijection.

Define $F^{\preceq\lambda} := \sum_{\mu\preceq\lambda} F^\mu$. By Z3 the right ideals $F^{\preceq\lambda} k\Omega$ are actually two-sided ideals. For any $k\Omega$-module $V$ define $V^{\preceq\lambda} := F^{\preceq\lambda} V$. This is a submodule of $V$ for all $\lambda$, $\lambda\preceq\mu \implies V^{\preceq\lambda}\leq V^{\preceq\mu}$ holds by construction and $F^\lambda$ acts as the identity on $V^{\preceq\lambda} / V^{\prec\lambda}$.

Now consider the equation $1=\sum_\lambda F^\lambda$. It shows that there must be at least one $\lambda$ with $F^\lambda V\neq 0$ if $V\neq 0$. A $\lambda$ that is $\preceq$-minimal with respect to this property satisfies $0\neq F^\lambda V = F^{\preceq\lambda} V$ so that $F^\lambda V=V$ follows if $V$ is simple. Therefore for each simple $k\Omega$-module $V$ there is exactly one $\lambda$ with $F^\lambda V=V$. Conversely $R^\lambda:=F^{\preceq\lambda} k\Omega / F^{\prec\lambda} k\Omega$ is a finite dimensional, non-zero $k\Omega$-module with $F^\lambda R^\lambda=R^\lambda$ so that for each $\lambda\in\Irr(W)$ there must be at least one simple $k\Omega$-module $V$ with $F^\lambda V=V$. This establishes another bijection between $\Irr(k\Omega)\to\Irr(W)$. By reindexing the $F^\lambda$ one can achieve that these two bijections are in fact the same so that $F^\lambda V=V$ holds if and only if the restriction of $k(v)V$ to $k(v)H$ is of isomorphism class $\lambda$.

\medbreak
Now consider again an arbitrary $V$ and the quotient $R^\lambda:=V^{\preceq\lambda}/V^{\prec\lambda}$. Because $F^\mu R^\lambda=0$ for all $\mu\neq\lambda$, the representation $k\Omega\to\operatorname{End}_k(W)$ must annihilate $F^\mu$ for all $\mu\neq\lambda$ and therefore all $F^{\kappa} X_{IJ} F^{\kappa'}$ with $\kappa\neq\kappa'$. Hence the representation vanishes on the radical and $R^\lambda$ is therefore semisimple. But again $F^\mu R^\lambda=0$ for all $\mu\neq\lambda$ so that the simple constituents of $R^\lambda$ must all lie in the isomorphism class $\lambda$.
\end{proof}

\section{Proving the decomposition conjecture}\label{section:proof_of_conjecture}

The rest of the paper is devoted to proving that the $W$-graph decomposition conjecture holds for Coxeter groups of types $I_2(m)$, $A_1$--$A_4$ and $B_3$. These proofs all proceed by the same pattern: The relations from Theorem \ref{theorem:Omega_relations} are used to find orthogonal decompositions $E_I=\sum_{\lambda\in\Irr(W)} F_I^\lambda$ of the vertex idempotents $E_I\in k\Omega$ into smaller idempotents $F_I^\lambda$ some of which may be zero. The idempotents $F^\lambda$ in the decomposition conjecture will then be obtained as $F^\lambda := \sum_I F_I^\lambda$.

These decompositions will be graphically represented as ``refinements'' of the compatibility graph $\mathcal{Q}_W$, i.e. the single vertex corresponding to $E_I$ will be ``split'' into up to $\abs{\Irr(W)}$ many vertices corresponding to the idempotents $F_I^\lambda$ (some of which might be zero) and similarly the edge corresponding to the element $X_{IJ}$ will be ``split'' into up to $\abs{\Irr(W)}^2$ many edges corresponding to the elements $F_I^\lambda X_{IJ} F_J^\mu$ most of which will also be zero.

Direct computations will be used to show that enough edge elements are zero to satisfy the decomposition conjecture.

\begin{remark}
A reviewer of this paper remarked that the computations in the rest of this paper feel like they are be instances of a general algorithmic approach to the question whether or not a particular Coxeter group satisfies the decomposition conjecture. I share this feeling, but to my frustration I have not been able to pin down such an algorithm and prove its correctness as of the time of this writing. Part of the complication stems from the fact that almost nothing useful about $\Omega$ is known to me in the absence of the decomposition conjecture. In particular it is hard to algorithmically decide whether or not an element is zero without having a nice, faithful representation of $\Omega$ at hand. Even proving finite dimensionality or even that the relations in Theorem \ref{theorem:Omega_relations} are a non-commutative Gröbner basis (and therefore the problem's amenability to certain general algorithms) is beyond my capabilities as of now.

If and when these problems get resolved, the lengthy calculations in this chapter may be replaced with a computer proof.
\end{remark}

\subsection{Auxiliary lemmas}

The first lemma which will be repeatedly used allows us to ``transport'' a decomposition into pairwise orthogonal idempotents from one $E_I$ to a adjacent $E_J$ in the compatibility graph and immediately recognize most of the possible new edge elements as zero.

\begin{definition}
In any algebra define a partial order on the set of idempotents by $e\leq f \iff e=ef=fe$.
\end{definition}

\begin{lemma}\label{lemma:idempotenttransport}
Let $I,J\subseteq S$ be arbitrary but fixed subsets. Let $A$ be a finite indexing set and $(e_\alpha)_{\alpha\in A}$ pairwise orthogonal idempotents $\leq E_I$ with $X_{IJ} X_{JI} = \sum_{\alpha\in A} \sigma_\alpha e_\alpha$ for some $\sigma_\alpha\in k^\times$. Denote the idempotent $E_I-\sum_\alpha e_\alpha$ by $e_0$. With these notations the following statements hold:
\begin{enumerate}
	\item $\widetilde{e}_\alpha := \sigma_\alpha^{-1} X_{JI} e_\alpha X_{IJ}$ and $\widetilde{e}_0:=E_J-\sum_{\alpha\in A} \widetilde{e}_\alpha$ are pairwise orthogonal idempotents $\leq E_J$.
	\item $X_{IJ} \widetilde{e}_\alpha = e_\alpha X_{IJ}$ and $X_{JI} e_\alpha = \widetilde{e}_\alpha X_{JI}$ for all $\alpha\in A\cup\Set{0}$.
	\item $r:=X_{JI} e_0 X_{IJ}$ satisfies $r^2=0$, $r = \widetilde{e}_0 r \widetilde{e}_0$ and $X_{JI} X_{IJ} = \sum_{\alpha\in A} \sigma_\alpha\widetilde{e}_\alpha+r$. In particular $r=0$ holds if $X_{JI} X_{IJ}$ is an idempotent itself.
	\item $X_{IJ} \widetilde{e}_\alpha X_{JI} = \sigma_\alpha e_\alpha$ for all $\alpha\in A$. In other words applying this construction twice gives back the original idempotents.
\end{enumerate}
\end{lemma}
\begin{proof}
All claims are easily verified by using the definition. For example
\begin{align*}
	\widetilde{e}_\alpha \widetilde{e}_\beta &= \sigma_\alpha^{-1} \sigma_\beta^{-1} X_{JI} e_\alpha X_{IJ} X_{JI} e_\beta X_{IJ} \\
	&= \sigma_\alpha^{-1} \sigma_\beta^{-1} X_{JI} e_\alpha (\sum_\gamma \sigma_\gamma e_\gamma) e_\beta X_{IJ} \\
	&= \sum_\gamma \frac{\sigma_\gamma}{\sigma_\alpha \sigma_\beta} X_{JI} e_\alpha e_\gamma e_\beta X_{IJ} \\
	&= \begin{cases} 
	0 & \alpha\neq \beta \\
	\widetilde{e}_\alpha & \alpha=\beta
	\end{cases}
\end{align*}
See \citep[Lemma 4.5.25]{hahn2013diss} for complete proofs of the other claims.
\end{proof}

\begin{definition}
In the above construction, the $\widetilde{e}_\alpha$ are said to be obtained by \emph{transporting idempotents from $I$ to $J$}. The $e_0$ and $\widetilde{e}_0$ are called \emph{leftover idempotents} of this transport.
\end{definition}

The following well-known result will also be used repeatedly to construct the morphisms $\psi_\lambda$ in conjecture Z4.
\begin{lemma}\label{lemma:matrix_alg_quiver}
The matrix algebra $k^{d\times d}$ is freely generated by the generators $\lbrace e_{ij} \mid 1\leq i,j\leq d, \abs{i-j}\leq 1\rbrace$ with respect to the relations
\[e_{ii} e_{jj} = \delta_{ij} e_{ii}, \quad 1=\sum_{i=1}^d e_{ii}, \quad e_{ii} e_{ij} e_{jj} = e_{ij} \quad\text{and}\quad e_{ij} e_{ji} = e_{ii}\]
\end{lemma}

Note that this can be equivalently stated by saying that $k^{d\times d}$ is the quotient of the path algebra of the quiver

\begin{figure}[h]
\centering
\begin{tikzpicture}
\node[font=\scriptsize] (E1) at (0,0) {$1$};
\node[font=\scriptsize] (E2) at (2,0) {$2$};

\node (E3) at (4,0) {$\cdots$};

\node[font=\scriptsize] (E4) at (6,0) {$d-1$};
\node[font=\scriptsize] (E5) at (8,0) {$d$};

\path[->,bend right=15,font=\scriptsize]
	(E1) edge (E2)
	(E2) edge (E1)
	(E2) edge (E3)
	(E3) edge (E2)
	(E3) edge (E4)
	(E4) edge (E3)
	(E4) edge (E5)
	(E5) edge (E4);
\end{tikzpicture}
\end{figure}

by the relations that declare every directed loop to be equal to (the idempotent corresponding to) its base point.

\bigbreak
While proving Z4 the surjectivity of the constructed morphisms will often be implied by the fact that $F^\lambda k\Omega F^\lambda$ is generated as a $k$-algebra by the elements $F_I^\lambda=F^\lambda E_I F^\lambda$ and $F^\lambda X_{IJ} F^\lambda$. This follows from the fact that $k\Omega$ is generated by the $E_I$ and $X_{IJ}$ together with the observation that Z1--Z3 implies that a product of the form
\[F^{\lambda_1} X_{I_1,I_2} \ldots X_{I_{k-1},I_k} F^{\lambda_k} = \sum_{\lambda_2,\ldots,\lambda_{k-1}} F^{\lambda_1} X_{I_1,I_2} F^{\lambda_2} \ldots F^{\lambda_{k-1}} X_{I_{k-1},I_k} F^{\lambda_k}\]
can only be non-zero if there are $\lambda_2,\ldots,\lambda_{k-1}$ with $F^{\lambda_j} X_{I_j,I_{j+1}} F^{\lambda_{j+1}}\neq 0$ for all $1\leq j<k$. By Z3 this implies that $\lambda_1\preceq\lambda_2\preceq\ldots\preceq\lambda_k$. Thus if $\lambda_1=\lambda_k=\lambda$, then all intermediate $\lambda_j$ must be equal to $\lambda$ as well so that $F^{\lambda} X_{I_1,I_2} \ldots X_{I_{k-1},I_k} F^{\lambda}$ is expressible as a product of elements of the form $F^\lambda X_{IJ} F^\lambda$ as claimed.

\subsection{Rank 1}

\begin{theorem}
The decomposition conjecture is true for all Coxeter groups $(W,S)$ of type $A_1\times\ldots\times A_1$.
\end{theorem}
\begin{proof}
Groups of this particular type have the property that all $s,t\in S$ commute. In particular there are no transversal edges in the compatibility graph but only inclusion edges so that $\mathcal{Q}_W$ is acyclic and the trivial decomposition $E_I=E_I$ is already sufficient to satisfy Z1 -- Z4.
%
%
%
%
\end{proof}

\subsection{Rank 2}

While good rings for $A_n$ and $B_n$ are easy to understand, the following lemma will be needed to establish the existence of certain elements in a good ring for Coxeter groups of $I_2(m)$ which will be used in the proof of the decomposition conjecture. Note that a good ring for $I_2(m)$ always contains $\IZ[2\cos(\tfrac{2\pi}{m}),\tfrac{1}{m}]$.

\begin{lemma}\label{lemma:cosines}
Let $m\in\IN_{\geq 1}$ and $k$ a good ring for $I_2(m)$. The following assertions are true:
\begin{enumerate}
	\item $2\cos(a\frac{2\pi}{m}), 4\cos(a\frac{\pi}{m})^2\in k$ for all $a\in\IZ$.
	\item $4\cos(a\frac{\pi}{m})^2\in k^\times$ for all $a\in\IZ\setminus \tfrac{m}{2}\IZ$.
	\item $4\cos(a\frac{\pi}{m})^2-4\cos(b\frac{\pi}{m})^2\in k^\times$ for all $1\leq a<b\leq\floor{\frac{m}{2}}$.
\end{enumerate}
\end{lemma}
\begin{proof}
Set $\zeta_n := \exp(\frac{2\pi i}{n})$ for all $n\in\IN_{\geq 1}$. With this notation $2\cos(a\tfrac{2\pi}{n}) = \zeta_n^a+\zeta_n^{-a}$ holds. It follows from $T^a+T^{-a}\in\IZ[T+T^{-1}]$ that $2\cos(a\tfrac{2\pi}{n})\in k$ for all $a\in\IZ$. The fact that $4\cos(a\frac{\pi}{m})^2\in k$ follows from the double-angle formula $2\cos(\tfrac{\theta}{2})^2 = \cos(\theta)+1$.

\medbreak
The proofs of the second and third claim use that
\[\IZ[2\cos(\tfrac{2\pi}{n})] = \IZ[\zeta_n+\zeta_n^{-1}] \subseteq \IZ[\zeta_n] \subseteq \IZ[\zeta_{nl}]\]
are integral ring extensions for all $l\in\IN_{\geq 1}$ and integral extensions $R\subseteq S$ have the property $R\cap S^\times = R^\times$. Therefore it suffices to show that the elements are units in $\IZ[\zeta_{ml}, \frac{1}{m}]$ for some $l\in\IN_{\geq 1}$.

\medbreak
Step 1: $4\cos(a\tfrac{\pi}{m})^2$ is invertible for all $a\in\IZ \setminus \tfrac{m}{2} \IZ$.

This follows from
\[\prod_{1\leq a<\tfrac{m}{2}} (2\cos(a\tfrac{\pi}{m}))^2 = \begin{cases}
	1 & \textrm{if}\,2\nmid m \\
	\tfrac{m}{2} & \textrm{if}\,2\mid m
	\end{cases}\]
which is easily shown using $2\cos(a\tfrac{\pi}{m})=\zeta_{2m}^a+\zeta_{2m}^{-a}$.
Therefore $4\cos(a\tfrac{\pi}{m})^2$ is invertible too.

\medbreak
Step 2: $2\sin(a\frac{\pi}{m})$ is invertible for all $a\in\IZ\setminus m\IZ$.

This follows from
\[\prod_{a=1}^{m-1} 2\sin(a\tfrac{\pi}{m}) = m\]
which similarly can be shown using $2\sin(a\tfrac{\pi}{m}) = \tfrac{1}{i}(\zeta_{2m}^a-\zeta_{2m}^{-a})$.
Hence all $2\sin(a\frac{\pi}{m})$ are units for $a\in\IZ\setminus m\IZ$. This then proves the third claim because $4\cos(a\tfrac{\pi}{m})^2 - 4\cos(b\tfrac{\pi}{m})^2 = 2\sin((a+b)\tfrac{\pi}{m})\cdot 2\sin((a-b)\tfrac{\pi}{m})$ holds.
\end{proof}

\begin{theorem}
Let $m$ be a natural number $\geq 3$. The decomposition conjecture is true for all Coxeter groups of type $I_2(m)$.
\end{theorem}
\begin{proof}
The idea of the proof is to use a ``spectral decomposition'' of the loops $X_{1,2}X_{2,1}$ and $X_{2,1}X_{1,2}$ and construct a refinement of the compatibility graph as in Figure \ref{fig:better_compatibility_graph_I2}.

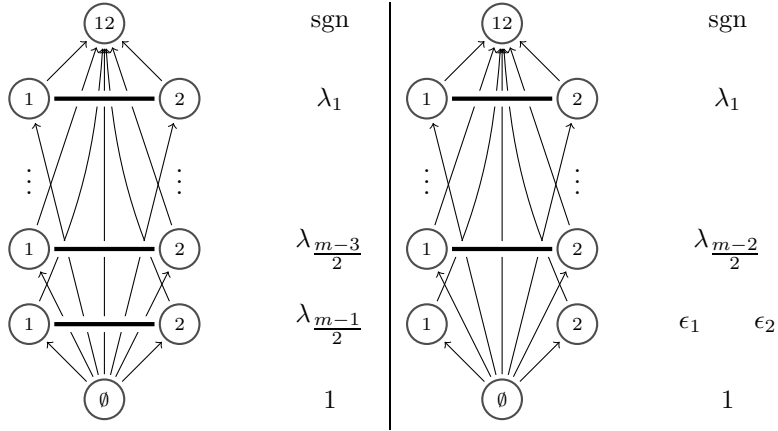
\begin{figure}[ht]
	\centering
	\begin{tabular}{c|c}
\begin{tikzpicture}[
		crossing line/.style = {preaction={draw=white,-,line width=6pt}}
	]

\node[Vertex] (E_empty) at (0,0) {$\emptyset$};

\node[Vertex] (F_1a) at (-1,1) {$1$};
\node[Vertex] (F_2a) at ( 1,1) {$2$};

\node[Vertex] (F_1b) at (-1,2) {$1$};
\node[Vertex] (F_2b) at ( 1,2) {$2$};

\node[Vertex] (F_1c) at (-1,4) {$1$};
\node[Vertex] (F_2c) at ( 1,4) {$2$};

\node[Vertex] (E_12) at (0,5) {$12$};

\node at (-1,3){$\vdots$};
\node at ( 1,3){$\vdots$};

\path[EdgeI]
	(E_empty)
		edge (F_1a)
		edge (F_2a)
		edge (F_1b)
		edge (F_2b)
		edge (F_1c)
		edge (F_2c)
		edge (E_12)
	(F_1a) edge[bend right=10] (E_12)
	(F_2a) edge[bend left=10] (E_12)
	(F_1b) edge (E_12)
	(F_2b) edge (E_12)
	(F_1c) edge (E_12)
	(F_2c) edge (E_12);
	
\path[EdgeT]
	(F_1a) edge[crossing line] (F_2a)
	(F_1b) edge[crossing line] (F_2b)
	(F_1c) edge[crossing line] (F_2c);

\node at (3,5) {$\sgn$};
\node at (3,4) {$\lambda_1$};
\node at (3,2) {$\lambda_{\tfrac{m-3}{2}}$};
\node at (3,1) {$\lambda_{\tfrac{m-1}{2}}$};
\node at (3,0) {$1$};

\end{tikzpicture}
		
		&
		
\begin{tikzpicture}[
	crossing line/.style = {preaction={draw=white,-,line width=6pt}}
	]

\node[Vertex] (E_empty) at (0,0) {$\emptyset$};

\node[Vertex] (F_1a) at (-1,1) {$1$};
\node[Vertex] (F_2a) at ( 1,1) {$2$};

\node[Vertex] (F_1b) at (-1,2) {$1$};
\node[Vertex] (F_2b) at ( 1,2) {$2$};

\node[Vertex] (F_1c) at (-1,4) {$1$};
\node[Vertex] (F_2c) at ( 1,4) {$2$};

\node[Vertex] (E_12) at (0,5) {$12$};

\node at (-1,3){$\vdots$};
\node at ( 1,3){$\vdots$};

\path[EdgeI]
	(E_empty)
		edge (F_1a)
		edge (F_2a)
		edge (F_1b)
		edge (F_2b)
		edge (F_1c)
		edge (F_2c)
		edge (E_12)
	(F_1a) edge[bend right=10] (E_12)
	(F_2a) edge[bend left=10] (E_12)
	(F_1b) edge (E_12)
	(F_2b) edge (E_12)
	(F_1c) edge (E_12)
	(F_2c) edge (E_12);
	
\path[EdgeT]
	(F_1b) edge[crossing line] (F_2b)
	(F_1c) edge[crossing line] (F_2c);

\node at (3,5) {$\sgn$};
\node at (3,4) {$\lambda_1$};
\node at (3,2) {$\lambda_{\tfrac{m-2}{2}}$};
\node at (2.5,1) {$\epsilon_1$};
\node at (3.5,1) {$\epsilon_2$};
\node at (3,0) {$1$};

\end{tikzpicture}

	\end{tabular}
	\caption{Refined compatibility graph for $I_2(m)$; left hand side for $m$ odd, right hand side for $m$ even.}
	\label{fig:better_compatibility_graph_I2}
\end{figure}

The next important observation is that there are only two transversal edges if the rank of $(W,S)$ is two, namely $X_{1,2}$ and $X_{2,1}$. Therefore the only relations in $\Omega$ of type $(\alpha)$ are
\[0 = \sum_{j=0}^{m-1} a_j \underbrace{X_{1,2} X_{2,1} \ldots}_{j}\quad\textrm{and}\quad 0 = \sum_{j=0}^{m-1} a_j \underbrace{X_{2,1} X_{1,2} \ldots}_{j}\]
where the $a_j$ are the coefficients of $\tau_{m-1}$.

\medbreak
Step 1: Preparations.

For all $n\in\IN$ define $\widetilde{\tau}_n\in\IZ[X]$ by
\[\widetilde{\tau}_n := \begin{cases} \tau_n(\sqrt{X}) & \textrm{if}\, 2\mid n \\ \tau_n(\sqrt{X})\sqrt{X} &\textrm{if}\, 2\nmid n\end{cases}.\]
Recall that $\tau_n$ is an even polynomial if $n$ is even and an odd polynomial if $n$ is odd. Therefore $\widetilde{\tau}_n$ really is a polynomial in $X$. It has degree $\ceil{\frac{n}{2}}$ and is monic. Since the $n$ zeros of $\tau_n$ are given by $2\cos(\frac{a}{n+1}\pi)$ for $a=1,\ldots,n$ (c.f. \citep[22.16]{abramowitz1964handbook}), the zeros of $\widetilde{\tau}_n$ are given by $4\cos(\tfrac{a}{n+1}\pi)^2$ for $a=1,\ldots,\ceil{\frac{n}{2}}$. In particular the zeros of $\widetilde{\tau}_{m-1}$ are equal to $\sigma_a:=4\cos(a\tfrac{\pi}{m})^2$ for $a=1,\ldots,\floor{\frac{m}{2}}$.

\medbreak
Step 2: Construction of the idempotents.

If $m$ is odd, then the $(\alpha)$-type relations are already of the form $\widetilde{\tau}_{m-1}(X_{1,2}X_{2,1})=0$ and $\widetilde{\tau}_{m-1}(X_{2,1}X_{1,2})=0$ respectively. If $m$ is even, then one can multiply the relation with $X_{1,2}$ and $X_{2,1}$ and obtain the same equations.

By defining
\begin{align*}
	F_{1,a} &:= \prod_{\substack{b=1,\ldots,\floor{\frac{m}{2}} \\ b\neq a}} \frac{X_{1,2} X_{2,1} - \sigma_b E_1}{\sigma_a - \sigma_b} \quad\textrm{and} \\
	F_{2,a} &:= \prod_{\substack{b=1,\ldots,\floor{\frac{m}{2}} \\ b\neq a}} \frac{X_{2,1} X_{1,2} - \sigma_b E_2}{\sigma_a - \sigma_b}
	\label{eq:I2:1}\tag{1}
\end{align*}
for all $a=1,\ldots,\floor{\frac{m}{2}}$ we get a set of pairwise orthogonal idempotents $F_{1,a}, F_{2,a}\in k\Omega$ with
\[E_1 = \sum_{a=1}^{\floor{\frac{m}{2}}} F_{1,a} \quad\textrm{and}\quad X_{1,2} X_{2,1} = \sum_{a=1}^{\floor{\frac{m}{2}}} \sigma_a F_{1,a} \quad\textrm{and}\]
\[E_2 = \sum_{a=1}^{\floor{\frac{m}{2}}} F_{2,a} \quad\textrm{and}\quad X_{2,1} X_{1,2} = \sum_{a=1}^{\floor{\frac{m}{2}}} \sigma_a F_{2,a}.\]

Denote the irreducible characters of $W(I_2(m))$ of degree two by $\lambda_a$ for $a=1,\ldots,\tfrac{m-1}{2}$ if $m$ is odd and $a=1,\ldots,\frac{m-2}{2}$ if $m$ is even. If $m$ is even there are two one-dimensional characters other than the trivial and the sign character which will be denoted by $\epsilon_1$ and $\epsilon_2$ respectively.

Now define the idempotents $(F^\lambda)_{\lambda\in\Irr(W)}$ as
\begin{align*}
	F^{1} &= E_\emptyset, \\
	F^{\lambda_a} &= F_{1,a}+F_{2,a}, \\
	F^{\operatorname{sgn}} &= E_{\Set{1,2}} \\
\intertext{and if $m$ is even, define further}
	F^{\epsilon_1} &= F_{1,\tfrac{m}{2}}\quad\textrm{and} \\
	F^{\epsilon_2} &= F_{2,\tfrac{m}{2}}.
\end{align*}
Now Z1 and Z2 hold by construction. It remains to verify Z3 and Z4.

\medbreak
Step 3: Proving Z3.

Now that we have the idempotents $F_{1,a}, F_{2,a}$ splitting $E_1$ and $E_2$ respectively, we can consider $\Omega$ as a quotient of the quiver which is obtained from $\mathcal{Q}_W$ by splitting the vertices labelled $\Set{1}$ and $\Set{2}$ into $\floor{\frac{m}{2}}$ vertices each. A priori this could lead to the edge elements $X_{1,2}$, $X_{2,1}$ being split into $\floor{\frac{m}{2}}^2$ new edge elements $F_{1,a} X_{1,2} F_{2,b}$ and $F_{2,a} X_{2,1} F_{1,b}$ respectively. We will show that this does not happen and instead all edge elements not depicted in Figure \ref{fig:better_compatibility_graph_I2} vanish.

This follows from Lemma \ref{lemma:idempotenttransport} because $F_{1,a}$ can be obtained from $F_{2,a}$ by idempotent transporting and vice versa: Note that $\sigma_a$ is invertible for $1\leq a<\frac{m}{2}$ and $\sigma_a=0$ for $a=\frac{m}{2}$. That means $F_{1,m/2}$ and $F_{2,m/2}$ are the leftover idempotents. The lemma for idempotent transporting can be applied. Now the following holds
\[\sum_{a=0}^{\floor{\tfrac{m}{2}}} \sigma_a X_{1,2} F_{2,a} X_{2,1} = X_{1,2} \Big(\underbrace{\sum_a \sigma_a F_{2,a}}_{=E_2}\Big) X_{2,1} = X_{1,2} X_{2,1} = \sum_{a=0}^{\floor{\tfrac{m}{2}}} \sigma_a F_{1,a}.\]
And because $X_{1,2} F_{2,a} X_{2,1}$ is an idempotent for $1\leq a<\frac{m}{2}$ both sides of the equation $\sum_a \sigma_a X_{1,2} F_{2,a} X_{2,1} = \sum_a \sigma_a F_{1,a}$ describe the spectral decomposition of $X_{1,2} X_{2,1}$. Since the $\sigma_a$ are pairwise distinct one obtains $F_{1,a} = X_{1,2} F_{2,a} X_{2,1}$ and for symmetry reasons $X_{2,1} F_{1,a} X_{1,2} = F_{2,a}$ for all $1\leq a<\frac{m}{2}$.

Now Lemma \ref{lemma:idempotenttransport} additionally implies $F_{1,a} X_{1,2} = X_{1,2} F_{2,a}$ so that $F_{1,a} X_{1,2} F_{2,b} = 0$ for $a\neq b$. And for symmetry reasons also $F_{2,a} X_{2,1} F_{1,b} = 0$ for $a\neq b$.

If $m$ is even, then it is also true that there are no edges $F_{1,m/2} \leftrightarrows F_{2,m/2}$. This can be seen as follows: By construction
\[\prod_{1\leq b<\frac{m}{2}} (X^2-\sigma_b) = \frac{\widetilde{\tau}_{m-1}(X^2)}{X^2} = \frac{\tau_{m-1}(X)}{X} = \sum_{j=1}^{m-1} a_j X^{j-1}\]
holds. By inserting $X_{2,1}X_{1,2}$ for $X^2$ and multiplying by $X_{1,2}$ this gives
\[X_{1,2} \prod_{1\leq b<\frac{m}{2}} (X_{2,1}X_{1,2}-\sigma_b) = \sum_{j=0}^{m-1} a_j \underbrace{X_{1,2} X_{2,1}\ldots}_{j\,\textrm{factors}} \overset{(\alpha)}{=} 0.\]
Now multiplication with the denominator of \eqref{eq:I2:1} gives $X_{1,2} F_{2,m/2} = 0$ so that there are no edges from $F_{2,m/2}$ to any vertex labelled with $\Set{1}$. And for symmetry reasons there can be no edge from $F_{1,m/2}$ to any vertex labelled with $\Set{2}$.

Therefore the only edges that can exist are edges $F_{1,a} \leftrightarrows F_{2,a}$, edges $\emptyset\to F_{i,a}$, the edge $\emptyset\to\Set{12}$ and edges $F_{i,a}\to\Set{12}$. This means Z3 is satisfied if we define a partial order on $\Irr(W)$ by declaring $\sgn$ as the top element, $1$ as the bottom element and all other elements as mutually incomparable.

\medbreak
Step 4: Proving Z4.

For the characters of degree 1 define $\psi_\lambda: k^{1\times 1}\to F^\lambda k\Omega F^\lambda$ by $\psi_\lambda(e_{11}):=F^\lambda$. This homomorphism is surjective because of the lack of closed loops based at $F^\lambda$ in the quiver displayed in Figure \ref{fig:better_compatibility_graph_I2}. Therefore $F^\lambda k\Omega F^\lambda = k\cdot F^\lambda$ holds and $\psi_\lambda$ is surjective.

For the characters of degree two define $\psi_{\lambda_a}: k^{2\times 2}\to F^{\lambda_a} k\Omega F^{\lambda_a}$ by
\[\begin{pmatrix} e_{11} & e_{12} \\ e_{21} & e_{22} \end{pmatrix} \mapsto \begin{pmatrix} F_{1,a} & F^{\lambda_a} X_{1,2} F^{\lambda_a} \\ \sigma_a^{-1} F^{\lambda_a} X_{2,1} F^{\lambda_a} & F_{2,a} \end{pmatrix}.\]
This is a well-defined algebra homomorphism by construction of $F^\lambda$. It is surjective because $F^{\lambda_a} k \Omega F^{\lambda_a}$ is generated by the elements $F_I^{\lambda_a}$ and $F^{\lambda_a} X_{IJ} F^{\lambda_a}$ all of which are contained in the image of $\psi_{\lambda_a}$.
\end{proof}

\subsection{Rank 3}

\begin{theorem}
The decomposition conjecture is true for type $A_3$.
\end{theorem}
We will not prove this in detail, since it is very similar to (although not formally a consequence of) the proof for $A_4$ which will be presented in the next section. Full details can also be found in \citep[Section 4.5]{hahn2013diss}.

\begin{theorem}
The decomposition conjecture is true for type $B_3$.
\end{theorem}
\begin{proof}
We aim for a refinement of the compatibility graph as depicted in Figure \ref{fig:gyoja:better_compatibility_graph_B3} (where inclusions edges were again omitted for the sake of clarity). The relations of type $(\alpha)$ will be crucial for that undertaking. We will write $(\alpha^{st})$ to denote that we have used the relation of type $(\alpha)$ belonging to the edge $s - t$ of the Dynkin diagram.

\begin{figure}[htp]
\centering
\begin{tikzpicture}

\node[Vertex] (E_empty) at (0,0) {$\emptyset$};

\node[Vertex] (E_0a) at (-1.5,1) {$0$};

\node[Vertex] (E_1a) at ( 0.0,1) {$1$};
\node[Vertex] (E_2a) at (+1.5,1) {$2$};

\node[Vertex] (E_0b) at (-1.5,2) {$0$};
\node[Vertex] (E_1b) at ( 0.0,2) {$1$};
\node[Vertex] (E_2b) at (+1.5,2) {$2$};

\node[Vertex] (E_0c) at (-1.5,3) {$0$};
\node[Vertex] (E_1c) at (-0.5,3) {$1$};
\node[Vertex] (E_02c)at (-0.5,5) {$02$};

\node[Vertex] (E_1d) at ( 0.5,3) {$1$};
\node[Vertex] (E_02d)at ( 0.5,5) {$02$};
\node[Vertex] (E_12d)at ( 1.5,5) {$12$};

\node[Vertex] (E_01b) at (-1.5,6) {$01$};
\node[Vertex] (E_02b) at ( 0.0,6) {$02$};
\node[Vertex] (E_12b) at ( 1.5,6) {$12$};

\node[Vertex] (E_01a) at (-1.5,7) {$01$};
\node[Vertex] (E_02a) at ( 0.0,7) {$02$};

\node[Vertex] (E_12a) at ( 1.5,7) {$12$};

\node[Vertex] (E_123) at (0,8) {$012$};

\path[EdgeT]
	(E_1a) edge (E_2a)
	
	(E_0b) edge (E_1b)
	(E_1b) edge (E_2b)
	
	(E_0c) edge (E_1c)
	(E_1c) edge (E_02c)
	
	(E_1d) edge (E_02d)
	(E_02d) edge (E_12d)
	
	(E_01b) edge (E_02b)
	(E_02b) edge (E_12b)
	
	(E_01a) edge (E_02a);

\node (lambda_3_empty)   at (6,0) {$\ydiagram{3},\emptyset$};
\node (lambda_empty_3)   at (5,1) {$\emptyset,\ydiagram{3}$};
\node (lambda_21_empty)  at (7,1) {$\ydiagram{2,1},\emptyset$};
\node (lambda_2_1)       at (6,2) {$\ydiagram{2},\ydiagram{1}$};
\node (lambda_1_2)       at (5,4) {$\ydiagram{1},\ydiagram{2}$};
\node (lambda_11_1)      at (7,4) {$\ydiagram{1,1},\ydiagram{1}$};
\node (lambda_1_11)      at (6,6) {$\ydiagram{1},\ydiagram{1,1}$};
\node (lambda_empty_21)  at (5,7) {$\emptyset,\ydiagram{2,1}$};
\node (lambda_111_empty) at (7,7) {$\ydiagram{1,1,1},\emptyset$};
\node (lambda_empty_111) at (6,8) {$\emptyset,\ydiagram{1,1,1}$};

\end{tikzpicture}
	\caption{Refined compatibility graph of $B_3$}
	\label{fig:gyoja:better_compatibility_graph_B3}
\end{figure}
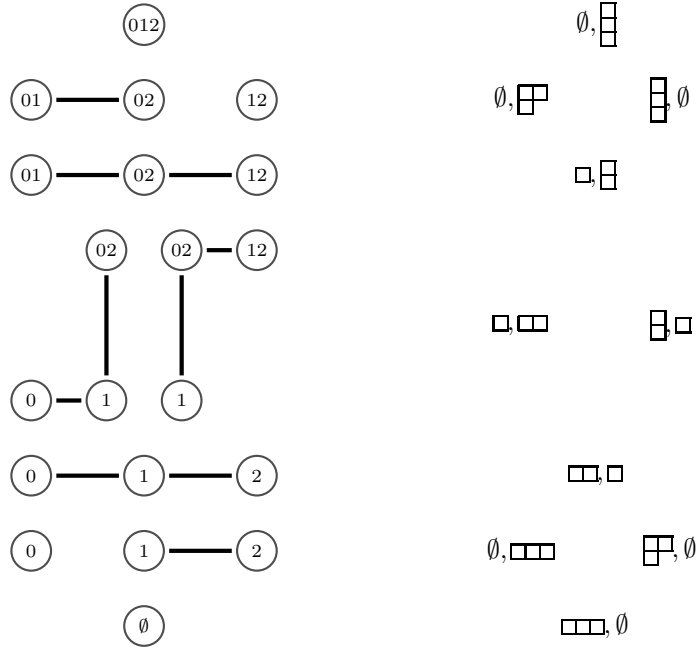

First note that every good ring for $B_3$ contains $\IZ[\tfrac{1}{2}]$ so that one is allowed to divide by two.

\medbreak
Step 1: Z1 and Z2.

We define elements $F_I^{\lambda,\mu}$ for all $I\subseteq S$ and all $(\lambda,\mu)\in\Irr(W)$ according to Table \ref{table:idempotents_B3} where absent entries are understood to be defined as zero. We will therefore prove that the $F_I^{\lambda,\mu}$ are pairwise orthogonal idempotents with $E_I=\sum_{\lambda,\mu} F_I^{\lambda,\mu}$.

\begin{table}[htp]
\setlength{\extrarowheight}{1em}
\ytableausetup{boxsize=0.3em}

\hspace{-10pt}
\footnotesize
\begin{tabular}{|MMMM|}
\hline

\multicolumn{4}{|M|}{E_{012}} \\

\hline

X_{01,02} F_{02}^{\emptyset,\ydiagram{2,1}} X_{02,01} & \multicolumn{2}{M|}{X_{02,01} X_{01,02} - F_{02}^{\ydiagram{1},\ydiagram{1,1}}} & E_{12} - F_{12}^{\ydiagram{1},\ydiagram{1,1}} - F_{12}^{\ydiagram{1,1},\ydiagram{1}} \\

\hline

X_{01,02} F_{02}^{\ydiagram{1},\ydiagram{1,1}} X_{02,01} & \multicolumn{2}{M}{\tfrac{1}{2} X_{02,12} X_{12,02}\cdot X_{02,01} X_{01,02}} & X_{12,02} F_{02}^{\ydiagram{1},\ydiagram{1,1}} X_{02,01} \\

\hline

& \multicolumn{1}{M|}{X_{02,1}X_{1,02} - F_{02}^{\ydiagram{1,1},\ydiagram{1}}} & X_{02,1} X_{1,02} \cdot X_{02,12} X_{12,02} & X_{12,02} F_{02}^{\ydiagram{1,1},\ydiagram{1}} X_{02,12} \\
X_{0,1} F_1^{\ydiagram{1},\ydiagram{2}} X_{1,0} & \multicolumn{1}{M|}{X_{1,0}X_{0,1} \cdot X_{1,02} X_{02,1}} & X_{1,02} X_{02,1} - F_1^{\ydiagram{1},\ydiagram{2}} & \\

\hline

X_{0,1} F_1^{\ydiagram{2},\ydiagram{1}} X_{1,0} & \multicolumn{2}{M}{\tfrac{1}{2} X_{1,0} X_{0,1} \cdot X_{1,02} X_{02,1}} & X_{2,1} F_1^{\ydiagram{2},\ydiagram{1}} X_{1,2} \\

\hline

E_0 - F_0^{\ydiagram{1},\ydiagram{2}} - F_0^{\ydiagram{2},\ydiagram{1}} & \multicolumn{2}{|M}{X_{1,0}X_{0,1} - F_1^{\ydiagram{1},\ydiagram{2}}} & X_{2,1} F_1^{\ydiagram{2},\ydiagram{1}} X_{1,2} \\

\hline

\multicolumn{4}{|M|}{E_\emptyset} \\

\hline
\end{tabular}
\caption{Expressions for the vertex idempotents of the refined compatibility graph of $B_3$ arranged in the same positions as the vertices in Figure \ref{fig:gyoja:better_compatibility_graph_B3}.}
\label{table:idempotents_B3}
\end{table}

The $(\alpha^{21})$-relation $E_2 = X_{2,1} X_{1,2}$ implies that $F_1' := X_{1,2} X_{2,1}$ is an idempotent $\leq E_1$. The $(\alpha^{12})$-relation
\[E_1 = X_{1,2} X_{2,1} + X_{1,02} X_{02,1}\]
implies that $F_1'':= X_{1,02} X_{02,1}$ also is an idempotent $\leq E_1$ which is orthogonal to $F_1'$. These two idempotents will be decomposed further.

\medbreak
Recall that relations of type $(\alpha)$ use the polynomials $\tau_{m-1}$ which for $m=4$ has the form $\tau_{4-1}(T) = T^3-2T$. Therefore $(\alpha^{01})$ and $(\alpha^{10})$ imply:
\begin{align}
	0 &= X_{0,1} X_{1,0} X_{0,1} + X_{0,1} X_{1,02} X_{02,1} - 2X_{0,1} \label{eq:gyoja_B3:1}\tag{1}\\
	0 &= X_{1,0} X_{0,1} X_{1,0} + X_{1,02} X_{02,1} X_{1,0} - 2X_{1,0} \label{eq:gyoja_B3:2}\tag{2}
\end{align}
By setting $f:=X_{1,0} X_{0,1}$, multiplying the first equation by $X_{1,0}$ from the left and the second with $X_{0,1}$ from the right we obtain:
\begin{align}
	0 &= f^2 + f F_1'' - 2f \label{eq:gyoja_B3:3}\tag{3}\\
	0 &= f^2 + F_1'' f - 2f \label{eq:gyoja_B3:4}\tag{4}
\end{align}
Thus
\[f'' := f F_1'' = F_1'' f \quad\text{and}\quad f':= f F_1' = F_1' f \]
are idempotents. We multiply \eqref{eq:gyoja_B3:3} with $F_1''$ and \eqref{eq:gyoja_B3:4} with $F_1'$ and obtain:
\begin{align}
	0 &= f''^2 - f'' \label{eq:gyoja_B3:5}\tag{5}\\
	0 &= f'^2 - 2f'  \label{eq:gyoja_B3:6}\tag{6}
\end{align}

This gives us the following decomposition into orthogonal idempotents:
\begin{align}
	E_1 = F_1' + F_1'' = \big(\underbrace{\tfrac{1}{2}f'}_{=F_1^{\ydiagram{2},\ydiagram{1}}}\big) + \big(\underbrace{F_1' - \tfrac{1}{2}f'}_{=F_1^{\ydiagram{2,1},\emptyset}}\big) + \big(\underbrace{f''}_{=F_1^{\ydiagram{1},\ydiagram{2}}}\big) + \big(\underbrace{F_1'' - f''}_{=F_1^{\ydiagram{1,1},\ydiagram{1}}}\big)
	\label{eq:gyoja_B3:idempotentsF1}\tag{7}
\end{align}
With these notations $X_{1,0}X_{0,1} = 2F_1^{\ydiagram{2},\ydiagram{1}} + F_1^{\ydiagram{1},\ydiagram{2}}$ holds. We see that the other idempotents are now related either by transporting of idempotents along $\Set{1}\to\Set{0}$ or $\Set{1}\to\Set{2}$ or by applying the antiautomorphism $\delta$ to previously constructed elements. In particular: The $F_I^{\lambda,\mu}$ defined in Table \ref{table:idempotents_B3} are pairwise orthogonal idempotents.

\bigbreak
Step 3: Verifying Z3.

We will check that in Figure \ref{fig:gyoja:better_compatibility_graph_B3} only ``upward'' edges appear so that the partial ordering on $\Irr(W)$ can be read off from the picture. We will in fact show that the only edges not depicted in Figure \ref{fig:gyoja:better_compatibility_graph_B3} are inclusion edges.

The following holds:
\begin{align*}
	X_{0,1} F_1^{\ydiagram{2,1},\emptyset} &= X_{0,1} (F_1' - \tfrac{1}{2} f') \\
	&= X_{0,1} (E_1 - \tfrac{1}{2} X_{1,0}X_{0,1}) F_1' \\
	&= \tfrac{1}{2} (2 X_{0,1} - X_{0,1} X_{1,0} X_{0,1}) F_1' \\
	&\overset{\text{\eqref{eq:gyoja_B3:1}}}{=} \tfrac{1}{2} (X_{0,1} F_1'') F_1' \\
	&= 0 \\
	X_{0,1} F_1^{\ydiagram{1,1},\ydiagram{1}} &= X_{0,1} (F_1'' - f'') \\
	&= X_{0,1} (E_1 - X_{1,0}X_{0,1}) F_1'' \\
	&= (X_{0,1} - X_{0,1} X_{1,0} X_{0,1}) F_1'' \\
	&\overset{\text{\eqref{eq:gyoja_B3:1}}}{=} (X_{0,1} F_1'' - X_{0,1}) F_1'' \\
	&= 0  
\end{align*}
This means that there cannot be edges from $F_1^{\ydiagram{2,1},\emptyset}$ or $F_1^{\ydiagram{1,1},\ydiagram{1}}$ to vertices labelled with $\Set{0}$. Since the idempotentes labelled by $\Set{0}$ were defined by transport of idempotents it follows from Lemma \ref{lemma:idempotenttransport} that $F_0^{\emptyset,\ydiagram{3}}X_{0,1} = 0$ holds, i.e. there are no edges from vertices labelled with $\Set{1}$ to $F_0^{\emptyset,\ydiagram{3}}$.

Analogously both $F_1^{\ydiagram{2,1},\emptyset} X_{1,0} = 0$ and $F_1^{\ydiagram{1,1},\ydiagram{1}} X_{1,0} = 0$ also hold.
Therefore there cannot be edges from vertices labelled with $\Set{0}$ to $F_1^{\ydiagram{2,1},\emptyset}$ or $F_1^{\ydiagram{1,1},\ydiagram{1}}$. Again it follows from \ref{lemma:idempotenttransport} that $X_{1,0}F_0^{\emptyset,\ydiagram{3}} = 0$ holds, i.e. there are no edges from $F_0^{\emptyset,\ydiagram{3}}$ to vertices labelled with $\Set{1}$.

Because the idempotents $F_0^{\ydiagram{1},\ydiagram{2}}$ and $F_0^{\ydiagram{2},\ydiagram{1}}$ were defined by transport of idempotents there are no edges $F_0^{\ydiagram{1},\ydiagram{2}} \leftrightarrows F_1^{\ydiagram{2},\ydiagram{1}}$ or $F_0^{\ydiagram{2},\ydiagram{1}} \leftrightarrows F_1^{\ydiagram{1},\ydiagram{2}}$. Similarly there are no edges $F_2^{\ydiagram{2,1},\emptyset} \leftrightarrows F_1^{\ydiagram{2},\ydiagram{1}}$ or $F_2^{\ydiagram{2},\ydiagram{1}} \leftrightarrows F_1^{\ydiagram{2,1},\emptyset}$.

Now we use the symmetry given by $\delta$ and obtain the same result for vertices labelled with $\Set{01},\Set{02},\Set{12}$.

It remains to verify that there are no edges ${F_1^{\ydiagram{1},\ydiagram{2}} \leftrightarrows F_{02}^{\ydiagram{1,1},\ydiagram{1}}}$ or ${F_1^{\ydiagram{1,1},\ydiagram{1}} \leftrightarrows F_{02}^{\ydiagram{1},\ydiagram{2}}}$. To this end we prove that the idempotents $F_{02}^{\ydiagram{1,1},\ydiagram{1}}$ and $F_{02}^{\ydiagram{1},\ydiagram{2}}$ are also given by an transport of idempotents. This follows from an application of the $(\alpha^{10})$-relation:
\begin{align*}
	0 &= X_{02,1} X_{1,0} X_{0,1} + X_{02,1} X_{1,02} X_{02,1} + X_{02,12} X_{12,02} X_{02,1} - 2X_{02,1}  \label{eq:gyoja_B3:8}\tag{8}
\end{align*}
Multiplying with $X_{1,02}$ from the left and using $X_{1,0}X_{0,1} = F_1^{\ydiagram{1},\ydiagram{2}} + 2 F_1^{\ydiagram{2},\ydiagram{1}}$ as well as $X_{02,12} X_{12,02} =F_{02}^{\ydiagram{1,1},\ydiagram{1}} + 2F_{02}^{\ydiagram{1},\ydiagram{1,1}}$ and $F_{02}'':=X_{02,1}X_{1,02}=\delta(F_1'')$ we obtain:
\begin{align*}
	0 &= X_{02,1} ( F_1^{\ydiagram{1},\ydiagram{2}} + 2F_1^{\ydiagram{2},\ydiagram{1}} ) X_{1,02} + F_{02}'' F_{02}'' + (F_{02}^{\ydiagram{1,1},\ydiagram{1}} + 2F_{02}^{\ydiagram{1},\ydiagram{1,1}})F_{02}'' - 2F_{02}'' \\
	&= X_{02,1} ( F_1^{\ydiagram{1},\ydiagram{2}} + 2F_1^{\ydiagram{2},\ydiagram{1}} ) X_{1,02} + F_{02}^{\ydiagram{1,1},\ydiagram{1}} F_{02}''+ 2\smash{\underbrace{F_{02}^{\ydiagram{1},\ydiagram{1,1}}F_{02}''}_{=0}} - F_{02}'' \\
	&= X_{02,1} ( F_1^{\ydiagram{1},\ydiagram{2}} + 2F_1^{\ydiagram{2},\ydiagram{1}} ) X_{1,02} + F_{02}^{\ydiagram{1,1},\ydiagram{1}} - F_{02}'' \\
	&= X_{02,1} ( F_1^{\ydiagram{1},\ydiagram{2}} + 2F_1^{\ydiagram{2},\ydiagram{1}} ) X_{1,02} + (-F_{02}^{\ydiagram{1},\ydiagram{2}})
\end{align*}
Hence we obtain
\[	F_{02}^{\ydiagram{1},\ydiagram{2}} = X_{02,1} F_1^{\ydiagram{1},\ydiagram{2}} X_{02,1} + 2\cdot X_{02,1} F_1^{\ydiagram{2},\ydiagram{1}} X_{1,02}.\]
The $(\alpha^{21})$-relation
\[ 0 = X_{02,1} X_{1,2}\]
implies $X_{02,1} F_1' = 0$ so that $X_{02,1} F_1^{\ydiagram{2},\ydiagram{1}} = 0$, because $F_1^{\ydiagram{2},\ydiagram{1}}\leq F_1'$. Therefore we obtain
\[ F_{02}^{\ydiagram{1},\ydiagram{2}} =  X_{02,1} F_1^{\ydiagram{1},\ydiagram{2}} X_{1,02}\]
i.e. $F_{02}^{\ydiagram{1},\ydiagram{2}}$ is a transported idempotent along the edge $\Set{02}\leftrightarrows\Set{1}$. Because of $F_1'' = F_1^{\ydiagram{1},\ydiagram{2}} + F_1^{\ydiagram{1,1},\ydiagram{1}}$ we also obtain
\[ F_{02}^{\ydiagram{1,1},\ydiagram{1}} =  X_{02,1} F_1^{\ydiagram{1,1},\ydiagram{1}} X_{1,02}\]
which together with Lemma \ref{lemma:idempotenttransport} implies that there are no edges other than the ones displayed in \ref{fig:gyoja:better_compatibility_graph_B3} between vertices labelled with $\Set{1}$ and $\Set{02}$. This shows that Z3 holds.

\bigbreak
Step 4: Verifying Z4.

There is not much to do for the characters of degree one. We define $\psi_{\lambda,\mu}: \IZ[\tfrac{1}{2}]^{1\times 1}\to F^{\lambda,\mu} \IZ[\tfrac{1}{2}]\Omega F^{\lambda,\mu}$ to be the only possible morphism, namely $\psi_{\lambda,\mu}(e_{11}):=F^{\lambda,\mu}$. The surjectivity of these maps is automatic because the four components for the one dimensional characters in the refined compatibility graph have no edges and therefore $F^{\lambda,\mu} \IZ[\tfrac{1}{2}]\Omega F^{\lambda,\mu} = \IZ[\tfrac{1}{2}] F^{\lambda,\mu}$.

\medbreak
\begin{table}[htp]
\ytableausetup{boxsize=0.3em}

\centering
\newcolumntype{L}{>{$}l<{$}}   

\footnotesize
\begin{tabular}{LMML}
\text{Character }\chi_{\lambda,\mu} & \text{Map }\psi_{\lambda,\mu} \\
\hline
\chi_{\ydiagram{2,1},\emptyset}  & 
	\begin{pmatrix}
	e_{11} & e_{12} \\ e_{21} & e_{22}	
	\end{pmatrix}&\mapsto&\begin{pmatrix}
		F_1^{\ydiagram{2,1},\emptyset} & X_{1,2}^{\ydiagram{2,1},\emptyset} \\
		X_{2,1}^{\ydiagram{2,1},\emptyset} & F_2^{\ydiagram{2,1},\emptyset}
	\end{pmatrix} \\
\chi_{\emptyset,\ydiagram{2,1}} &
	\begin{pmatrix}
	e_{11} & e_{12} \\ e_{21} & e_{22}	
	\end{pmatrix}&\mapsto&\begin{pmatrix}
		F_{02}^{\emptyset,\ydiagram{2,1}} & -X_{02,01}^{\emptyset,\ydiagram{2,1}} \\
		-X_{01,02}^{\emptyset,\ydiagram{2,1}} & F_{01}^{\emptyset,\ydiagram{2,1}}
	\end{pmatrix} \\
\chi_{\ydiagram{2},\ydiagram{1}} &
	\begin{pmatrix}
	e_{11} & e_{12} & \\ e_{21} & e_{22} & e_{23} \\ & e_{32} & e_{33}
	\end{pmatrix}&\mapsto&\begin{pmatrix}
		F_0^{\ydiagram{2},\ydiagram{1}} & X_{0,1}^{\ydiagram{2},\ydiagram{1}} & \\
		\tfrac{1}{2}X_{1,0}^{\ydiagram{2},\ydiagram{1}} & F_1^{\ydiagram{2},\ydiagram{1}} & X_{1,2}^{\ydiagram{2},\ydiagram{1}} \\
		& X_{2,1}^{\ydiagram{2},\ydiagram{1}} & F_2^{\ydiagram{2},\ydiagram{1}}
	\end{pmatrix} \\
\chi_{\ydiagram{1},\ydiagram{1,1}} &
	\begin{pmatrix}
	e_{11} & e_{12} & \\ e_{21} & e_{22} & e_{23} \\ & e_{32} & e_{33}
	\end{pmatrix}&\mapsto&\begin{pmatrix}
		F_{12}^{\ydiagram{1},\ydiagram{1,1}} & -X_{12,02}^{\ydiagram{1},\ydiagram{1,1}} & \\
		-\tfrac{1}{2}X_{02,12}^{\ydiagram{1},\ydiagram{1,1}} & F_{02}^{\ydiagram{1},\ydiagram{1,1}} & -X_{02,01}^{\ydiagram{1},\ydiagram{1,1}} \\
		& -X_{01,02}^{\ydiagram{1},\ydiagram{1,1}} & F_{01}^{\ydiagram{1},\ydiagram{1,1}}
	\end{pmatrix} \\
\chi_{\ydiagram{1},\ydiagram{2}} &
	\begin{pmatrix}
	e_{11} & e_{12} & \\ e_{21} & e_{22} & e_{23} \\ & e_{32} & e_{33}
	\end{pmatrix}&\mapsto&\begin{pmatrix}
		F_0^{\ydiagram{1},\ydiagram{2}} & X_{0,1}^{\ydiagram{1},\ydiagram{2}} & \\
		-X_{1,0}^{\ydiagram{1},\ydiagram{2}} & F_1^{\ydiagram{1},\ydiagram{2}} & X_{1,02}^{\ydiagram{1},\ydiagram{2}} \\
		& X_{02,1}^{\ydiagram{1},\ydiagram{2}} & F_{02}^{\ydiagram{1},\ydiagram{2}}
	\end{pmatrix} \\
\chi_{\ydiagram{1,1},\ydiagram{1}} &
	\begin{pmatrix}
	e_{11} & e_{12} & \\ e_{21} & e_{22} & e_{23} \\ & e_{32} & e_{33}
	\end{pmatrix}&\mapsto&\begin{pmatrix}
		F_{12}^{\ydiagram{1,1},\ydiagram{1}}  & X_{12,02}^{\ydiagram{1,1},\ydiagram{1}} & \\ 
		X_{02,12}^{\ydiagram{1,1},\ydiagram{1}} & F_{02}^{\ydiagram{1,1},\ydiagram{1}}  & -X_{02,1}^{\ydiagram{1,1},\ydiagram{1}} \\
		& -X_{1,02}^{\ydiagram{1,1},\ydiagram{1}} & F_1^{\ydiagram{1,1},\ydiagram{1}} 
	\end{pmatrix}
\end{tabular}
\caption{Morphisms $\psi_{\lambda,\mu} : \IZ[\tfrac{1}{2}]^{d_{\lambda,\mu}\times d_{\lambda,\mu}} \twoheadrightarrow F^{\lambda,\mu} \IZ[\tfrac{1}{2}]\Omega F^{\lambda,\mu}$ for $2\leq d_{\lambda,\mu}\leq 3$}
\label{tab:inverses_B3}
\end{table}

In table \ref{tab:inverses_B3} are all the morphisms $\psi_{\lambda,\mu}: \IZ[\tfrac{1}{2}]^{d_{\lambda,\mu}\times d_{\lambda,\mu}} \to F^{\lambda,\mu} \IZ[\tfrac{1}{2}]\Omega F^{\lambda,\mu}$ for the characters of degree two and three where we use the notation $X_{IJ}^{\lambda,\mu}:=F^{\lambda,\mu} X_{IJ} F^{\lambda,\mu}$.

We have used again that $\IZ[\tfrac{1}{2}]^{d\times d}$ is the $\IZ[\tfrac{1}{2}]$-algebra given by the presentation in Lemma \ref{lemma:matrix_alg_quiver}. These relations are satisfied by construction of the $F^{\lambda,\mu}$ and therefore all the maps in the table are well-defined algebra morphisms.

\medbreak
The construction of $\psi_{\lambda,\mu}$ ensures that all idempotents $F_I^{\lambda,\mu}$ for all $I\subseteq S$ and all $F^{\lambda,\mu} X_{IJ} F^{\lambda,\mu}$ for transversal edges $I\leftrightarrows J$ are contained in the image of $\psi_{\lambda,\mu}$. For $(\lambda,\mu) \in\Set{ (\emptyset,\ydiagram{2,1}),(\ydiagram{2,1},\emptyset),(\ydiagram{2},\ydiagram{1}),(\ydiagram{1},\ydiagram{1,1})}$ this is already enough the guarantee surjectivity, because all edges in the component of $F^{\lambda,\mu}$ are transversal edges.

\medbreak
For $\chi_{\ydiagram{1},\ydiagram{2}}$ on the other hand there could be an inclusion edge $\Set{0}\to\Set{0,2}$ and for $\chi_{\ydiagram{1,1},\ydiagram{1}}$ there could be an inclusion edge $\Set{1}\to\Set{1,2}$. To complete the proof we show that this is not the case by using the $(\beta^{20})$-relation:
\begin{align*}
	X_{02,0}^{\ydiagram{1},\ydiagram{2}} &= F^{\ydiagram{1},\ydiagram{2}} X_{02,0} F_0^{\ydiagram{1},\ydiagram{2}} \\
	&= F^{\ydiagram{1},\ydiagram{2}} X_{02,0} (X_{0,1} F_1^{\ydiagram{1},\ydiagram{2}} X_{1,0}) \\
	&= F^{\ydiagram{1},\ydiagram{2}} (X_{02,0} X_{0,1}) F_1^{\ydiagram{1},\ydiagram{2}} X_{1,0} \\
	&\overset{\mathclap{(\beta^{20})}}{=} F^{\ydiagram{1},\ydiagram{2}} (X_{02,2} X_{2,1} + X_{02,12}X_{12,1} - X_{02,01} X_{01,1}) F_1^{\ydiagram{1},\ydiagram{2}} X_{1,0} \\
	&= (X_{02,2}^{\ydiagram{1},\ydiagram{2}} X_{2,1}^{\ydiagram{1},\ydiagram{2}} + X_{02,12}^{\ydiagram{1},\ydiagram{2}} X_{12,1}^{\ydiagram{1},\ydiagram{2}} - X_{02,01}^{\ydiagram{1},\ydiagram{2}} X_{01,1}^{\ydiagram{1},\ydiagram{2}}) F_1^{\ydiagram{1},\ydiagram{2}} X_{1,0}
\end{align*}
All summands within the brackets disappear because the $\ydiagram{1},\ydiagram{2}$ component in the graph \ref{fig:gyoja:better_compatibility_graph_B3} has no vertices labelled $\Set{2}$, $\Set{12}$ or $\Set{01}$. Using the symmetry given by $\delta$ the equation $F_{12}^{\ydiagram{1,1},\ydiagram{1}} X_{12,1} = 0$ also holds.
\end{proof}

\subsection{Rank 4}

\begin{theorem}
The decomposition conjecture is true for type $A_4$.
\end{theorem}

\begin{proof}
We will use an analogous strategy as before and use the relations of type $(\alpha)$. Again we will write $(\alpha^{st})$ to denote that we have used the relation of type $(\alpha)$ belonging to the edge $s - t$ of the Dynkin diagram and similarly for $(\beta)$-type relations.

Our goal is to decompose the compatibility graph as in Figure \ref{fig:gyoja:better_compatibility_graph_A4} (inclusion edges have been omitted for the sake of clarity).

\medbreak
Step 1: Verifying Z1 and Z2.

We will define idempotents $F_I^\lambda\leq E_I$ for all $I\subseteq S, \lambda\in\Irr(W)$ and set $F^\lambda := \sum_I F_I^\lambda$. First note that by lemma \ref{lemma:matrix_alg_quiver} the idempotents of a matrix algebra are given by evaluating loops in the quiver. Looking at the quiver we want to arrive at in Figure \ref{fig:gyoja:better_compatibility_graph_A4}, we therefore define the idempotents $F_I^\lambda$ either as one of the $E_I$ at the boundary of the compatibility graph or as loops connecting inner vertices to those outer vertices. More precisely we will use the definitions in table \ref{table:idempotents_A4} where all $F_I^\lambda$ not appearing there are understood to be defined as zero.

\begin{table}[htp]
\ytableausetup{boxsize=0.7em}

\hspace{-35pt}\footnotesize
\begin{tabular}{C{180pt}ccc}
 & $\lambda$ & $I$ & $F_I^\lambda$ \\
\hline

\begin{tikzpicture}[every node/.style={scale=0.7}]
\node[Vertex] (E_1234) at (0,9.0) {$1234$};	
\end{tikzpicture} & \ydiagram{1,1,1,1,1} & $\Set{1,2,3,4}$ & $E_{1234}$ \\

\cline{2-4}

\begin{tikzpicture}[every node/.style={scale=0.7}]
\node[Vertex] (E_123e) at (-3.0,8.0) {$123$};
\node[Vertex] (E_124e) at (-1.0,8.0) {$124$};
\node[Vertex] (E_134e) at (+1.0,8.0) {$134$};
\node[Vertex] (E_234e) at (+3.0,8.0) {$234$};
	
\path[EdgeT]
	(E_123e) edge (E_124e)
	(E_124e) edge (E_134e)
	(E_134e) edge (E_234e);
\end{tikzpicture} & \ydiagram{2,1,1,1} &
$\begin{array}{c}
\Set{1,2,3} \\
\Set{1,2,4} \\
\Set{1,3,4} \\
\Set{2,3,4}
\end{array}$ & $\begin{array}{c}
E_{123} \\
X_{124,123} X_{123,124} \\
X_{134,234} X_{234,134} \\
E_{234}
\end{array}$ \\

\cline{2-4}

\begin{tikzpicture}[every node/.style={scale=0.7}]
\node[Vertex] (E_124d) at (-1.0,7.0) {$124$};
\node[Vertex] (E_13d)  at (-1.0,5.5) {$13$};
\node[Vertex] (E_23d)  at ( 0.0,6.0) {$23$};
\node[Vertex] (E_24d)  at (+1.0,5.5) {$24$};
\node[Vertex] (E_134d) at (+1.0,7.0) {$134$};

\path[EdgeT]
	(E_124d) edge (E_13d)
	(E_13d) edge (E_23d)
	(E_23d) edge (E_24d)
	(E_24d) edge (E_134d);
\end{tikzpicture} & \ydiagram{2,2,1} &
$\begin{array}{c}
\Set{1,2,4} \\
\Set{1,3} \\
\Set{2,3} \\
\Set{2,4} \\
\Set{1,3,4} \end{array}$ & $\begin{array}{c}
X_{124,13} X_{13,124} \\
X_{13,124}X_{124,13} \\
X_{23,13}X_{13,124}X_{124,13}X_{13,23} \\
X_{24,134}X_{134,24} \\
X_{134,24}X_{24,134}
\end{array}$ \\

\cline{2-4}

\begin{tikzpicture}[every node/.style={scale=0.7}]
\node[Vertex] (E_12c) at (-3.0,4.5) {$12$};
\node[Vertex] (E_13c) at (-1.0,4.5) {$13$};
\node[Vertex] (E_14c) at ( 0.0,4.0) {$14$};
\node[Vertex] (E_23c) at ( 0.0,5.0) {$23$};
\node[Vertex] (E_24c) at (+1.0,4.5) {$24$};
\node[Vertex] (E_34c) at (+3.0,4.5) {$34$};

\path[EdgeT]
	(E_13c) edge (E_23c) edge (E_14c) edge (E_12c)
	(E_24c) edge (E_23c) edge (E_14c) edge (E_34c);
\end{tikzpicture} & \ydiagram{3,1,1} &
$\begin{array}{c}
 \Set{1,2} \\
 \Set{1,3} \\
 \Set{2,3} \\
 \Set{1,4} \\
 \Set{2,4} \\
 \Set{3,4}
\end{array}$ & $\begin{array}{c}
E_{12} \\
X_{13,12}X_{12,13} \\
X_{23,13}X_{13,12}X_{12,13}X_{13,23} \\
X_{14,24}X_{24,34}X_{34,24}X_{24,14} \\
X_{24,34}X_{34,24} \\
E_{34}
\end{array}$ \\

\cline{2-4}

\begin{tikzpicture}[every node/.style={scale=0.7}]
\node[Vertex] (E_2b)  at (-1.0,2.0) {$2$};
\node[Vertex] (E_13b) at (-1.0,3.5) {$13$};
\node[Vertex] (E_14b) at ( 0.0,3.0) {$14$};
\node[Vertex] (E_24b) at (+1.0,3.5) {$24$};
\node[Vertex] (E_3b)  at (+1.0,2.0) {$3$};

\path[EdgeT]
	(E_2b) edge (E_13b)
	(E_13b) edge (E_14b)
	(E_14b) edge (E_24b)
	(E_24b) edge (E_3b);
\end{tikzpicture} & \ydiagram{3,2} &
$\begin{array}{c}
\Set{2} \\
\Set{1,3} \\
\Set{1,4} \\
\Set{2,4} \\
\Set{3}
\end{array}$ & $\begin{array}{c}
X_{2,13}X_{13,2} \\
X_{13,2}X_{2,13} \\
X_{14,13}X_{13,2}X_{2,13}X_{13,14} \\
X_{24,3}X_{3,24} \\
X_{3,24}X_{24,3}
\end{array}$ \\

\cline{2-4}

\begin{tikzpicture}[every node/.style={scale=0.7}]
\node[Vertex] (E_1a) at (-3.0,1) {$1$};
\node[Vertex] (E_2a) at (-1.0,1) {$2$};
\node[Vertex] (E_3a) at (+1.0,1) {$3$};
\node[Vertex] (E_4a) at (+3.0,1) {$4$};

\path[EdgeT]
	(E_1a) edge (E_2a)
	(E_2a) edge (E_3a)
	(E_3a) edge (E_4a);
\end{tikzpicture} & \ydiagram{4,1} &
$\begin{array}{c}
\Set{1} \\
\Set{2} \\
\Set{3} \\
\Set{4}
\end{array}$ & $\begin{array}{c}
E_1 \\
X_{2,1} X_{1,2} \\
X_{3,4} X_{4,3} \\
E_4
\end{array}$ \\

\cline{2-4}

\begin{tikzpicture}[every node/.style={scale=0.7}]
\node[Vertex] (E_empty) at (0,0) {$\emptyset$};
\end{tikzpicture} & \ydiagram{5} & $\emptyset$ & $E_\emptyset$ \\

\hline

\end{tabular}
\captionof{figure}{The refined compatbility graph for $A_4$}
\label{fig:gyoja:better_compatibility_graph_A4}

\captionof{table}{Vertex idempotents of the refined compatibility graph for $A_4$}
\label{table:idempotents_A4}
\end{table}

\ytableausetup{boxsize=0.3em}
Once these elements have been defined, we have to prove that they are in fact idempotents and $E_I=\sum_\lambda F_I^\lambda$ is an orthogonal decomposition. Z2 will then be satisfied because $F^\lambda E_I = F_I^\lambda = E_I F^\lambda$ holds by definition.

The $(\alpha^{12})$-relation implies
\begin{align*}
E_1 &= X_{1,2} X_{2,1}
\end{align*}
from which it follows that $F_2^{\ydiagram{4,1}}$ is an idempotent $\leq E_2$, namely the idempotent obtained by transport of $E_1\leq E_1$ along the edge $\Set{1} \rightarrow \Set{2}$. From the $(\alpha^{12})$-relation
\[E_2 = X_{2,1} X_{1,2} + X_{2,13} X_{13,2}\]
we deduce that $F_2^{\ydiagram{3,2}}$ is the leftover idempotent of this transport. By applying the non-trivial graph automorphism we obtain that $F_3^{\ydiagram{4,1}}$ and $F_{24}^{\ydiagram{3,2}}$ are idempotents as well and by applying the antiautomorphism $\delta$ we find that $F_{134}^{\ydiagram{2,1,1,1}}$, $F_{13}^{\ydiagram{2,2,1}}$, $F_{124}^{\ydiagram{2,1,1,1}}$ and $F_{2,4}^{\ydiagram{2,2,1}}$ are idempotents too. And because Lemma \ref{lemma:idempotenttransport} also gives us orthogonality with the leftover idempotent, we're done done with all except the two-element subsets of $S$.

\medbreak
By transporting $F_2^{\ydiagram{3,2}}=X_{2,13} X_{13,2}$ along $\Set{2}\to\Set{13}$ we obtain the idempotent
\begin{align*}
	X_{13,2} F_2^{\ydiagram{3,2}} X_{2,13} &= X_{13,2} E_2 X_{2,13} - X_{13,2} F_2^{\ydiagram{4,1}} X_{2,13} \\
	&= X_{13,2} X_{2,13} - X_{13,2}X_{2,1} \smash{\underbrace{X_{1,2} X_{2,13}}_{=0 \text{ by }(\alpha^{12})}} \\
	&= F_{13}^{\ydiagram{3,2}}
\end{align*}

By applying $\delta$ and the graph automorphism we find that $F_{13}^{\ydiagram{2,2,1}}$, $F_{24}^{\ydiagram{3,2}}$ and $F_{134}^{\ydiagram{2,2,1}}$ are also idempotents.

\medbreak
The $(\alpha^{23})$-relation
\[E_{12} = X_{12,13} X_{13,12}\]
implies that $F_{13}^{\ydiagram{3,1,1}}$ is the idempotent obtained by transporting $E_{12}$ along $\Set{12}\to\Set{13}$.

Considering the $(\alpha^{32})$-relations
\begin{align*}
E_{13} &= X_{13,2} X_{2,13} + X_{13,12} X_{12,13} + X_{13,124} X_{124,13} = F_{13}^{\ydiagram{3,2}} + F_{13}^{\ydiagram{3,1,1}} + F_{13}^{\ydiagram{2,2,1}} \\
0 &= X_{2,13} X_{13,12} = X_{2,13} X_{13,124} \\
0 &= X_{12,13} X_{13,2} = X_{12,13} X_{13,124} \\
0 &= X_{124,13} X_{13,2} = X_{124,13} X_{13,12}
\end{align*} 
we find that $F_{13}^{\ydiagram{3,2}}, F_{13}^{\ydiagram{3,1,1}}, F_{13}^{\ydiagram{2,2,1}}$ constitute a orthogonal decomposition of $E_{13}$.

By applying the graph automorphism we find that $F_{24}^{\ydiagram{3,1,1}}$, $F_{24}^{\ydiagram{3,2}}$ and $F_{24}^{\ydiagram{2,2,1}}$ are pairwise idempotents as well.

\medbreak
We are now almost done. We need still need to look at the inner most vertices $\Set{2,3}$ and ${1,4}$ of the compatibility graph. The following $(\alpha^{34})$-relation holds:
\[E_{13} = X_{13,14} X_{14,13} + X_{13,124} X_{124,13}\]
This means $X_{13,14}X_{14,13} = F_{13}^{\ydiagram{3,2}}+ F_{13}^{\ydiagram{3,1,1}}$. Transporting these two idempotents along $\Set{13}\to\Set{14}$ we obtain $F_{14}^{\ydiagram{3,2}}$ and $F_{14}^{\ydiagram{3,1,1}}$. By symmetry $F_{23}^{\ydiagram{3,1,1}}$and $F_{23}^{\ydiagram{2,2,1}}$ are idempotents as well.

From the $(\alpha^{43})$- and the $(\alpha^{21})$-relation
\[E_{14} = X_{14,13} X_{13,14} \quad\text{and}\quad E_{23} = X_{23,13} X_{13,23}\]
we can infer that the two leftover idempotents for these transports vanish. Therefore we get orthogonal decompositions $E_{14} = F_{14}^{\ydiagram{3,2}} + F_{14}^{\ydiagram{3,1,1}}$ and $E_{23} = F_{23}^{\ydiagram{3,1,1}} + F_{23}^{\ydiagram{2,2,1}}$.

\bigbreak
Step 2: Verifying Z3.

We will prove that the only edges between the components not displayed in Figure \ref{fig:gyoja:better_compatibility_graph_A4} are inclusion edges from which it follows that the dominance ordering on $\lbrace \lambda \vdash 5\rbrace$ is the sought-after partial ordering. Because we have constructed all idempotents by transport of idempotents, most transversal edges split into parallel edges. That eliminates almost all possible transversal edges between different components.

\medbreak
The $(\alpha^{32})$-relation
\[X_{13,2} X_{2,3} = 0\]
implies that $F_2^{\ydiagram{3,2}} X_{2,3} = 0$ so that there is no transversal edge emanating from $E_3=F_3^{\ydiagram{4,1}}+ F_3^{\ydiagram{3,2}}$ and going to $F_2^{\ydiagram{3,2}}$. By symmetry there are no transversal edges going from $E_2$ to $F_3^{\ydiagram{3,2}}$ which shows that there are only inclusion edges between the $\ydiagram{4,1}$ and the $\ydiagram{3,2}$ component. Applying $\delta$ we find the same between the $\ydiagram{2,2,1}$ and the $\ydiagram{2,1,1,1}$ component.

\medbreak
The only other possibility are transversal edges of the form $\Set{14} \leftrightarrows \Set{24}$ and $\Set{23} \leftrightarrows \Set{24}$ because we have not used idempotent transport along these edges. Instead we worked with $\Set{13} \leftrightarrows \Set{14}$ and $\Set{13} \leftrightarrows \Set{23}$.

\medbreak
Consider the $(\beta^{24})$-relation
\[X_{24,23} X_{23,13}+ X_{24,2} X_{2,13} = X_{24,14}X_{14,13} + X_{24,134}X_{134,13}\]
which implies
\begin{align*}
F_{24}^{\ydiagram{3,2}} \cdot X_{24,14} \cdot F_{14}^{\ydiagram{3,1,1}} &= F_{24}^{\ydiagram{3,2}} \cdot \underbrace{X_{24,14} \cdot X_{14,13}} F_{13}^{\ydiagram{3,1,1}} X_{13,14} \\
&\overset{\mathclap{(\beta^{24})}}{=} F_{24}^{\ydiagram{3,2}} \cdot (-X_{24,134}X_{134,13}+X_{24,23}X_{23,13}+X_{24,2}X_{2,13}) F_{13}^{\ydiagram{3,1,1}} X_{24,14} \\
&=  - X_{24,3}\underbrace{X_{3,24} X_{24,134}}_{=0 \text{ by }(\alpha^{32})} X_{134,13} F_{13}^{\ydiagram{3,1,1}} X_{24,14} \\
&\phantom{=} + X_{24,3} \underbrace{X_{3,24} X_{24,23}}_{=0 \text{ by }(\alpha^{43})} X_{23,13} F_{13}^{\ydiagram{3,1,1}} X_{24,14} \\
&\phantom{=} + X_{24,3}X_{3,24} X_{24,2} \underbrace{X_{2,13} F_{13}^{\ydiagram{3,1,1}}}_{=0} X_{24,14} \\
&= 0
\end{align*}
Similarly combining the $(\beta)$- and $(\alpha)$-relations one shows that the transversal edge $\Set{14}\to\Set{24}$ splits into a pair of parallel edges as displayed in Figure \ref{fig:gyoja:better_compatibility_graph_A4} and all four of the possible cross-component edges are indeed zero. Applying $\delta$ we find the same holds between the $\ydiagram{3,1,1}$ and the $\ydiagram{2,2,1}$ component.

\medbreak
This shows that even for the edges $\Set{14} \leftrightarrows \Set{24}$ and $\Set{23} \leftrightarrows \Set{24}$ the idempotents on both sides are given by transporting idempotents and hence there can only be parallel edges as depicted in Figure \ref{fig:gyoja:better_compatibility_graph_A4}. The only other possible edges are those not depicted in this picture in other words the inclusion edges.

\bigbreak
Step 3: Verifying Z4.

We will construct surjective homomorphisms $\psi_\lambda: \IZ^{d_\lambda\times d_\lambda} \to F^\lambda \Omega F^\lambda$. We will use the presentation of $\IZ^{d_\lambda\times d_\lambda}$ from Lemma \ref{lemma:matrix_alg_quiver}.

We will use the abbreviation $X_{IJ}^\lambda := F^\lambda X_{IJ} F^\lambda$. By construction the equation $X_{IJ}^\lambda X_{JI}^\lambda = F_I^\lambda$ holds for all transversal edges $I\leftrightarrows J$ if $F_I^\lambda$ and $F_J^\lambda$ are both non-zero as well as $X_{IJ}^\lambda = X_{JI}^\lambda = 0$ otherwise.

\medbreak
If we denote the vertices of a component in Figure \ref{fig:gyoja:better_compatibility_graph_A4} with its index set (which is possible without conflicts since no index set occurs more than once), then
\[\psi_\lambda:\IZ^{d_\lambda\times d_\lambda}\to F^\lambda \Omega F^\lambda, e_{II} \mapsto F_{I}^\lambda, e_{IJ} \mapsto X_{IJ}^\lambda\]
defines a morphism $\psi_\lambda: \IZ^{d_\lambda\times d_\lambda}\to F^\lambda \Omega F^\lambda$ for those components which are straight lines without their inclusion edges, that is all components except the one labelled with $\lambda=\ydiagram{3,1,1}$.

\medbreak
We will prove surjectivity of $\psi_\lambda$ which is equivalent to showing that all $X_{IJ}^\lambda$ are contained in the image $\psi_\lambda$. For the transversal edges this is clear from the construction. Therefore we are done for $\lambda=\ydiagram{5},\ydiagram{4,1},\ydiagram{2,1,1,1}$ and $\ydiagram{1,1,1,1,1}$.

For $\lambda=\ydiagram{3,2}$ we must consider the inclusion edges $X_{13,3}$ and $X_{24,2}$. We use the relations of type $(\beta)$:
\begin{align*}
	X_{24,2}^{\ydiagram{3,2}} &= F_{24}^{\ydiagram{3,2}} \cdot X_{24,2} \cdot F_2^{\ydiagram{3,2}} \\
	&= X_{24,3} X_{3,24} \cdot (X_{24,2} \cdot X_{2,13}) X_{13,2} \\
	&\overset{\mathclap{(\beta^{42})}}{=} X_{24,3} X_{3,24} (X_{24,14} X_{14,13} + X_{24,134} X_{134,13} -X_{24,23} \underbrace{X_{23,13}) X_{13,2}}_{=0} \\
	&= X_{24,3} X_{3,24} X_{24,14} X_{14,13} X_{13,2} + X_{24,3} \underbrace{X_{3,24} X_{24,134}}_{=0} X_{134,13} X_{13,2} & \\
	&= X_{24,3}^{\ydiagram{3,2}} X_{3,24}^{\ydiagram{3,2}} X_{24,14}^{\ydiagram{3,2}} X_{14,13}^{\ydiagram{3,2}} X_{13,2}^{\ydiagram{3,2}} \quad\text{because } X_{24,2}^{\ydiagram{3,2}}\in F^{\ydiagram{3,2}}\Omega F^{\ydiagram{3,2}}\\
	&\in\im(\psi_{\ydiagram{3,2}})
\end{align*}
By applying the graph automorphism we obtain $X_{13,3}^{\ydiagram{3,2}}\in\im(\psi_{\ydiagram{3,2}})$ and by applying the antiautomorphism $\delta$ we obtain $X_{124,24}^{\ydiagram{2,2,1}},X_{134,13}^{\ydiagram{2,2,1}}\in\im(\psi_{\ydiagram{2,2,1}})$. Therefore all $X_{IJ}^\lambda$ are contained in the image of $\psi_\lambda$ for $\lambda=\ydiagram{3,2}, \ydiagram{2,2,1}$ and surjectivity holds in both cases.

\medbreak
It remains to handle the case $\lambda=\ydiagram{3,1,1}$. We sort the two element sets in the following order: $\Set{1,2}$, $\Set{1,3}$, $\Set{1,4}$, $\Set{2,3}$, $\Set{2,4}$, $\Set{3,4}$ and claim that the following homorphism ${\psi_{\ydiagram{3,1,1}}: \IZ^{6\times 6}\to F^{\ydiagram{3,1,1}}\Omega F^{\ydiagram{3,1,1}}}$ is well-defined:
\[\begin{pmatrix}
	e_{11} & e_{12} & & & & \\
	e_{21} & e_{22} & e_{23} & & & \\
	& e_{32} & e_{33} & e_{34} & & \\
	& & e_{43} & e_{44} & e_{45} & \\
	& & & e_{54} & e_{55} & e_{56} \\
	& & & & e_{65} & e_{66}
\end{pmatrix} \mapsto \left(\begin{smallmatrix}
	F_{12}^\lambda & X_{12,13}^\lambda  & & & & \\
	X_{13,12}^\lambda & F_{13}^\lambda & X_{13,14}^\lambda & & & \\
	& X_{14,13}^\lambda & F_{14}^\lambda & X_{14,13}^\lambda X_{13,23}^\lambda & & \\
	& & X_{23,13}^\lambda X_{13,14}^\lambda &  F_{23}^\lambda & X_{23,24}^\lambda & \\
	& & & X_{24,23}^\lambda & F_{24}^\lambda & X_{24,34}^\lambda \\
	& & & & X_{34,24}^\lambda & F_{34}^\lambda
\end{smallmatrix}\right)\]
Most relations from Lemma \ref{lemma:matrix_alg_quiver} are satisfied by construction of the idempotents. We still need to verify
\[X_{14,13}^\lambda X_{13,23}^\lambda \cdot X_{23,13}^\lambda X_{13,14}^\lambda = F_{14}^\lambda \quad\text{and}\]
\[X_{23,13}^\lambda X_{13,14}^\lambda \cdot X_{14,13}^\lambda X_{13,23}^\lambda = F_{23}^\lambda.\]
These equations follow from the $(\alpha)$-relations $E_{13}=X_{13,23}X_{23,13}$ and $E_{24}=X_{24,14}X_{14,24}$.

We verify the surjectivity of $\psi_\lambda$. By construction most $X_{IJ}^\lambda$ are already contained in the image. We only have to consider the edges between $F_{14}^\lambda \leftrightarrows F_{24}^\lambda$ and $F_{13}^\lambda \leftrightarrows F_{23}^\lambda$.
\begin{align*}
	X_{23,13}^\lambda &= X_{23,13}^\lambda F_{13}^\lambda \\
	&= X_{23,13}^\lambda (X_{13,14}^\lambda X_{14,13}^\lambda) \\
	&= (X_{23,13}^\lambda X_{13,14}^\lambda) X_{14,13}^\lambda \in \im(\psi_\lambda)
\end{align*}
And
\begin{align*}
X_{24,14}^\lambda &= X_{24,14}^\lambda F_{14}^\lambda \\
&=X_{24,14}^\lambda (X_{14,13}^\lambda X_{13,14}^\lambda) \\
&=(X_{24,14}^\lambda X_{14,13}^\lambda) X_{13,14}^\lambda \\
&\overset{\mathclap{(\beta)}}{=} (X_{24,23}^\lambda X_{23,13}^\lambda) X_{13,14}^\lambda \\
&= X_{24,23}^\lambda (X_{23,13}^\lambda X_{13,14}^\lambda) \in\im(\psi_\lambda)
\end{align*}
Applying $\delta$ we find $X_{13,23}^\lambda, X_{14,24}^\lambda\in\im(\psi_\lambda)$ as well. Therefore $\psi_\lambda$ is surjective.

\end{proof}


\bibliography{wgrph}

\end{document}